\documentclass[11pt, reqno, twoside, makeidx]{amsart}
\usepackage[margin=2.7cm, marginpar=1cm]{geometry}

\usepackage{hyperref}
\usepackage[english]{babel}
\usepackage[utf8]{inputenc}
\usepackage[T1]{fontenc}
\usepackage{amsmath,amsthm,amssymb,amsfonts}
\numberwithin{equation}{section}
\usepackage{enumerate}
\usepackage{faktor}
\usepackage{dsfont}
\usepackage{scalerel,stackengine}
\usepackage{textcomp}
\usepackage{graphicx}
\usepackage{extarrows}
\usepackage{epigraph}
\usepackage{graphicx}
\usepackage{subcaption}
\usepackage{sidecap}
\usepackage{indentfirst}
\usepackage{tikz}
\usepackage{tikz-cd}
\usepackage{tkz-euclide}
\usetikzlibrary{shapes.misc,arrows,matrix,patterns,decorations.markings,positioning}
\tikzset{cross/.style={cross out, draw=black, minimum size=2*(#1-\pgflinewidth), inner sep=0pt, outer sep=0pt},
cross/.default={4.5pt}}
\usepackage{pgfplots}
\usepackage{caption}
\usepackage{float}
\usepackage{color}
\usepackage{epstopdf}
\usepackage{epsfig}
\usepackage{stmaryrd}
\usepackage{color}
\usepackage{geometry}
\usepackage{multicol}
\usepackage{verbatim}
\usepackage{wrapfig}

\DeclareMathOperator{\Ker}{Ker }
\DeclareMathOperator{\Imm}{Im }
\DeclareMathOperator{\tb}{tb}
\DeclareMathOperator{\rot}{rot}
\DeclareMathOperator{\self}{sl}
\DeclareMathOperator{\lk}{\ell k}

\DeclareMathOperator{\Spin}{Spin}

\DeclareMathOperator{\Id}{Id}

\renewcommand{\geq}{\geqslant}
\renewcommand{\leq}{\leqslant} 
\renewcommand{\epsilon}{\varepsilon}

\newcommand{\Z}{\mathbb{Z}}
\newcommand{\Q}{\mathbb{Q}}
\newcommand{\F}{\mathbb{F}}

\newcommand{\C}{\mathbb C}

\newcommand{\xist}{\xi_{\text{st}}}

\newcommand{\x}{\mathbf{x}}
\newcommand{\y}{\mathbf{y}}
\newcommand{\e}{\mathbf{e}}
\newcommand{\s}{\mathfrak{s}}
\newcommand{\uu}{\mathfrak{u}}

\newtheorem{teo}{Theorem}[section]
\newtheorem*{teo*}{Theorem}
\newtheorem{lemma}[teo]{Lemma}
\newtheorem{prop}[teo]{Proposition}
\newtheorem*{prop*}{Proposition}
\newtheorem{defin}{Definition}
\newtheorem{cor}[teo]{Corollary}

\newtheorem{quest}[teo]{Question}
\newtheorem{remark}[teo]{Remark}

\usepackage{xpatch}
\makeatletter
\xpatchcmd{\@thm}{\thm@headpunct{.}}{\thm@headpunct{}}{}{}
\makeatother

\pgfplotsset{compat=1.18}
\begin{document}
\title[Holomorphic curves in Stein domains and the tau-invariant]{Holomorphic curves in Stein domains and the tau-invariant}
\author{Antonio Alfieri and Alberto Cavallo}
\address{Centre de recherches math\'ematiques (CRM), Montr\'eal (QC) H3C 3J7, Canada}
\email{antonioalfieri90@gmail.com}
\address{Institute of Mathematics of the Polish Academy of Sciences (IMPAN), Warsaw 00-656, Poland}
\email{acavallo@impan.pl}
\subjclass[2020]{57K18, 57K33, 57K43}



\vspace{-0.3cm}
\begin{abstract}
 The scope of the paper is threefold. First, we build on recent work by Hayden to compute Hedden's tau-invariant $\tau_{\xi}(L)$ in the case when $\xi$ is a Stein fillable contact structure on a rational homology sphere, and $L$ is a transverse link arising as the boundary of a pseudo-holomorphic curve. This leads to a new proof of the relative Thom conjecture for Stein domains.  Secondly, we compare the invariant $\tau_\xi$ to the Grigsby-Ruberman-Strle
 topological tau-invariant $\tau_\s$, associated to the $\Spin^c$-structure $\s=\s_\xi$ of the contact structure $\xi$, to obtain topological obstructions for a link type to admit a holomorphically fillable transverse representative. 
 Finally, we use our main result together with methods from lattice cohomology to compute the $\tau_\s$-invariants of certain links in lens spaces, and estimate their PL slice genus.
\end{abstract}
\maketitle

\thispagestyle{empty}

\section{Introduction}
A knot $K$ in $S^3$ is called slice if it bounds a smooth disk $\Delta$ in $D^4$. The question of characterising slice knots was first asked by Fox in 1985 and lead to a lot of interesting mathematics in recent years.

To obstruct a knot from being slice one can look at some classical invariants like the Arf invariant or the Tristram-Levine signatures. If these invariants do not work, one can resort to sophisticated knot invariants like the Ozsv\' ath-Szab\' o $\tau$-invariant \cite{Ozsvath}, and  the Rasmussen $s$-invariant \cite{Rasmussen}. These invariants contain a lot of information, but unlike classical invariants they can be extremely hard to compute. Nevertheless, computations can be performed for some particular classes of knots and links with nice combinatorics. This is the case of alternating knots \cite{MO}, links of plane curves singularities \cite{GN}, and \emph{quasi-positive} links. 
 
Recall that a link  $L$ in $S^3$ is called quasi-positive if it can be expressed as the closure of an $n$-braid $\beta$ that has a factorisation of the form:
\[\beta = (w_1\sigma_{j_1}w_1^{-1})\cdot...\cdot(w_b\sigma_{j_b}w_b^{-1}) \ ,\]  
 where the $\sigma_i$'s are the Artin generators of the $n$-braid group. It was shown \cite{Hedden-+,Cavallo-tau,Cavallo-qp} that the identity
\begin{equation}
 \tau(L) = \frac{ w(\beta)-n +\left|  L \right| }{2}
 \label{eq:qp}
\end{equation} 
holds for any quasi-positive link $L$ in $S^3$, where $|L|$ denotes the number of components of the link $L$, and $w(\beta)$ denotes the writhe of $\beta$, which is defined as the difference between the number of positive and negative generators in the factorisation of $\beta$. Quasi-positive links are exactly those links that arise as boundary of holomorphic curves properly embedded in $D^4$ \cite{BO,Rudolph}. 
This paper is devoted to the study of links that can be realised as boundary of pseudo-holomorphic curves in Stein domains.


\subsection*{Rational slice genus} To obstruct a knot $K$ in $S^3$ from being slice one can look at its branched double-cover $\Sigma(K)$. If the knot $K$ is slice then $\Sigma(K)$ bounds a rational homology ball $X$, namely  the branched double-cover of $D^4$ along a slice disk. This observation has been used by Lisca \cite{Lisca} to confirm the slice-ribbon conjecture for many interesting classes of knots. His work uses Donaldson's theorem and relies on some heavy, but elementary arithmetics.

One of the main limitations of Lisca's technique is that, for reasons that are not related to the topology of the knot $K$, the branched double-cover $\Sigma(K)$ may bound a rational homology ball. This is the case for the Conway knot for example, whose double-cover can be obtained by surgery on a slice knot. To obtain further information one can look at the pull-back knot $\widetilde{K} \subset \Sigma(K)$. Indeed, if a knot is slice, not only $\Sigma(K)=\partial X$ for some rational homology ball $X$ but one can also find a smooth disk $\Delta \subset X$ with $\partial \Delta =\widetilde{K}$. This leads to the following question.

\begin{quest}\label{rationalyslice} Suppose that $K$ is a knot, in a rational homology sphere $Y$, and that $Y=\partial X$ for some rational homology ball $X$. Then is it possible to find a smooth disk $\Delta \subset X$ with $\partial \Delta =K$? 
\end{quest}

If the answer to Question \ref{rationalyslice} is affirmative for a knot $K$ in $Y$ we say that $K$ is \emph{rationally slice} in $X$. Of course if $Y=\partial X$ for some rational homology ball $X$ we can consider the obvious notion of genus:
\[g^X_4(K)= \min_F\left\{g(F) : F\subset X \text{ oriented surface, } \partial F =K\right\} , \]
and define $g_4(K)= \displaystyle\min_X \left\{g^X_4(K) : X  \text{ rational homology ball with } \partial X= Y \right \}$.

\begin{remark}
One can also consider the notion of slice genus of links in rational homology spheres. In the case of links we use the negative of the Euler characteristics $-\chi(F)$,  to measure the complexity of surfaces. We define 
\[\chi_4^X(L,\Sigma)= \max_F\left\{\chi(F) : F\subset X \text{ oriented surface, } \partial F =L, [F]=\Sigma\in H_2(X,\partial X;\Z)\right\} \ .\]
If $X$ is a rational homology ball then the slice genus of $L$ in $X$ is denoted by $\chi_4^X(L)$, and again we set $\chi_4(L)$ to be the maximum of $\chi_4^X(L)$ over all rational balls $X$ bounding the underlying three-manifold. 
\end{remark}

In \cite{GRS} Grigsby, Ruberman and Strle proposed to study Question \ref{rationalyslice} with the methods of Heegaard Floer homology. They associate to a knot $K$ in a rational homology sphere $Y$, equipped with a $\Spin^c$-structure $\s$, a numerical invariant $\tau_\s(K)$, and they show that if $K$ bounds a smooth disk in a rational homology ball $X$ then $\tau_\s(K)=0$ for all $\Spin^c$-structures $\s$ extending over $X$. A refined definition of this invariant (which clarified an ambiguity in \cite{GRS}) has been given by Raoux \cite{Raoux}. Another topological $\tau$-invariant can be defined by using cobordism maps in Heegaard Floer homology: when $(Y,\s)$ bounds a suitable negative-definite four-manifold $(X,\uu)$, such that $\uu\lvert_Y=\s$, we have the distinguished non-vanishing homology class $\widehat F_{W,\mathfrak u}(\textbf e_0)\in\widehat{HF}(Y,\s)$, which allows us to introduce a different version of tau.

While this approach towards answering Question \ref{rationalyslice} sounds very promising it has been very little explored. One of the main issues is that there are not many computational techniques available. We note that Celoria \cite{Celoria} developed a software based on grid homology\footnote{Available at  \url{https://sites.google.com/view/danieleceloria/programs/grid-homology}.} that can be used to compute the  $\tau_\s$-invariant of knots and links in lens spaces. 
Furthermore, the first author \cite{Alfieri} found a formula based on lattice cohomology which applies to certain knots in almost rational plumbed three-manifolds. In this paper we perform computations in the case of quasi-positive links, and investigate how the invariants $\tau_\mathfrak{s}$ relate to contact geometry.  

\subsection*{Main results} 
If $(M,\xi)$ is a contact three-manifold then the homotopy class of  $\xi$ specifies a canonical $\Spin^c$-structure $\s:=\s_\xi$ and thus a somewhat canonical tau-invariant $\tau_\s$. Alternatively, Hedden \cite{Hedden} constructed another link invariant $\tau_\xi$. This is defined using the contact invariant $\widehat c(\xi)$ of Ozsv\'ath and Szab\'o \cite{OSz-contact}, and unlike $\tau_\s$ depends on the geometry of $\xi$: homotopic contact plane distributions may have different $\tau_\xi$-invariants. In addition, if $(M,\xi)$ is the boundary of a Stein domain $(W,J)$ (or more generally, of a strong symplectic filling) then we can also consider the invariant $\tau_\theta$, induced by the cobordism map $\widehat F_{W,J}:\widehat{HF}(S^3)\rightarrow\widehat{HF}(Y,\s)$. We compare the definition of $\tau_\mathfrak{s},\:\tau_\xi$ and $\tau_\theta$ in Section \ref{section:two} below. 

The contact invariant $\widehat c(\xi)$ is sent to $c^+(\xi)$ by the map $\widehat{HF}(-Y,\s)\rightarrow HF^+(-Y,\s)$; moreover, we can compose this map with the projection $HF^+(-Y,\s)\rightarrow HF^+_{\text{red}}(-Y,\s)$. When $c^+(\xi)$ is vanishing in $HF^+_{\text{red}}(-Y,\s)$ then it follows from Heegaard Floer theory that the invariant can be identified with the distinguished homology class $\Theta^+\in HF^+(-Y,\s)$, see \cite[Section 2]{MT} for details. Our main result can be stated as follows.


\begin{teo} \label{teo:main}
 Suppose that $M$ is a rational homology sphere equipped with a contact structure $\xi$, and let $(W,J)$ be a Stein filling of $(M,\xi)$. If $T$ is a transverse link in $(M,\xi)$ which is the boundary of a properly embedded  pseudo-holomorphic curve $C\subset W$ then
 \[ 2\tau_\xi(T)-|T| = \self_\Q(T) = -\chi(C) - [C]\cdot[C] + c_1(J)[C]\:,\]
 where $\self_\Q(T)$ denotes the rational self-linking number of $T$ in $(M,\xi)$. 
\end{teo}
As a corollary we get a new proof of the Relative Thom conjecture for Stein domains.

\begin{teo}[Relative Thom conjecture for Stein fillings]  \label{teo:relative}
 If $(W,J)$ is a Stein filling of a rational homology sphere then a properly embedded pseudo-holomorphic curve in $W$ maximises the Euler characteristic within its relative homology class.
\end{teo}

The version for symplectic curves was proved in \cite{GK}.  Their proof builds on the proof of the symplectic Thom conjecture for closed manifolds \cite{OSz} and uses the existence of symplectic caps \cite{Eliashberg-caps}. Another proof was suggested by Kronheimer based on some work of Mrowka and Rollin \cite{MR}. Our proof is based on Heegaard Floer homology instead, and is closer in spirit to Rasmussen's original proof of the Milnor conjecture \cite{Rasmussen}. The main ingredients of our proof of Theorem \ref{teo:main} are some recent work by Hayden \cite{Hayden} that we combine with ideas  by Hedden, Plamenevskaya, and the second author \cite{Hedden-+,Plamenevskaya,Cavallo-qp}.

Note that Theorem \ref{teo:main} also implies the following result.
\begin{prop}
 \label{prop:relative}
 Under the assumptions of Theorem \ref{teo:main}, if $\theta=\widehat F_{W,J}(\textbf e_0)\in\widehat{HF}(M,\s)$ and $\s$ is the $\Spin^c$-structure induced by $\xi$, then $\tau_\xi(T)=\tau_\theta(T)$. Furthermore, if $c^+(\xi)=\Theta^+$ then $\tau_\xi(T)=\tau_\s(T)$.   
\end{prop}
We state the following corollary explicitly since it was the conjecture that originally motivated our work.

\begin{cor}\label{cor:intro}
 Let $L$ be a link in a three-manifold $M$ bounding a rational homology ball $W$. Suppose that $W$ has a Stein structure $J$, and that $L$ bounds a properly embedded pseudo-holomorphic curve in $W$. Then one has \[2\tau_\xi(L)-|L|=-\chi^W_4(L)\:.\] In addition, if $c^+(\xi)=\Theta^+$ then $\tau_\s(L)$ is the maximal tau-invariant associated to the $\Spin^c$-structures of $M$ extending over $W$, where $\s$ is specified by the complex tangencies of $J$.
\end{cor}

Our corollaries are based on the inequality in Theorem \ref{teo:upper_bound} below. 

\begin{teo}\label{teo:upper_bound}
 Let $T$ be a transverse link in a contact three-manifold $(M,\xi)$. If $(M,\xi)$ admits a strong symplectic filling $(W,\omega)$
 then:
 \[\self_\Q(T) \leq 2 \tau_\theta(T) -|T|  \ , \] where $\theta=\widehat F_{W,J}(\textbf e_0)\in\widehat{HF}(M,\s)$ and $\s,J$ denote the $\Spin^c$-structures induced by $\xi$ and $\omega$. 
\end{teo}
A version of this result when $c^+(\xi)$ is vanishing in $HF_{\text{red}}(-M,\s)$ was proved by Mark and Tosun in \cite{MT}.



The methods of this paper can also be used to find topological obstructions for a link type to be the boundary of a pseudo-holomorphic curve in some Stein fillings. We state two corollaries in this direction. In the following two statements, $L$ denotes a link in a contact three-sphere $(M,\xi)$ with $c^+(\xi)$ vanishing in $HF_{\text{red}}(-M,\s)$, and $\s$ the $\text{Spin}^c$-structure associated to $\xi$. 

\begin{cor}\label{obs2}
 Suppose $H_1(M;\Z)$ does not contain a metaboliser $G$ such that $|\tau_{\s+\alpha}(L)|\leq \tau_\s(L)$, for every $\alpha \in G$. Then $L$ does not bound a pseudo-holomorphic curve in any rational homology ball Stein filling of $(M,\xi)$.
\end{cor} 

The following  corollary also stroke our attention.  
\begin{cor}\label{obs1} 
 Suppose $(W,J)$ is a rational homology ball Stein filling of $(M,\xi)$, and that $\tau_{\s}(L)\neq\tau_{\overline{\s}}(L)$. Then $L$ bounds a pseudo-holomorphic curve in $W$, with respect to at most one of the two Stein structures $J$ and $-J$.
\end{cor}

\begin{figure}[t]
\vspace{0.3cm}
 \centering
  \def\svgwidth{11.5cm}
\begingroup%
  \makeatletter%
  \providecommand\color[2][]{%
    \errmessage{(Inkscape) Color is used for the text in Inkscape, but the package 'color.sty' is not loaded}%
    \renewcommand\color[2][]{}%
  }%
  \providecommand\transparent[1]{%
    \errmessage{(Inkscape) Transparency is used (non-zero) for the text in Inkscape, but the package 'transparent.sty' is not loaded}%
    \renewcommand\transparent[1]{}%
  }%
  \providecommand\rotatebox[2]{#2}%
  \newcommand*\fsize{\dimexpr\f@size pt\relax}%
  \newcommand*\lineheight[1]{\fontsize{\fsize}{#1\fsize}\selectfont}%
  \ifx\svgwidth\undefined%
    \setlength{\unitlength}{575.25201272bp}%
    \ifx\svgscale\undefined%
      \relax%
    \else%
      \setlength{\unitlength}{\unitlength * \real{\svgscale}}%
    \fi%
  \else%
    \setlength{\unitlength}{\svgwidth}%
  \fi%
  \global\let\svgwidth\undefined%
  \global\let\svgscale\undefined%
  \makeatother%
  \begin{picture}(1,0.29876515)%
    \lineheight{1}%
    \setlength\tabcolsep{0pt}%
    \put(0,0){\includegraphics[width=\unitlength,page=1]{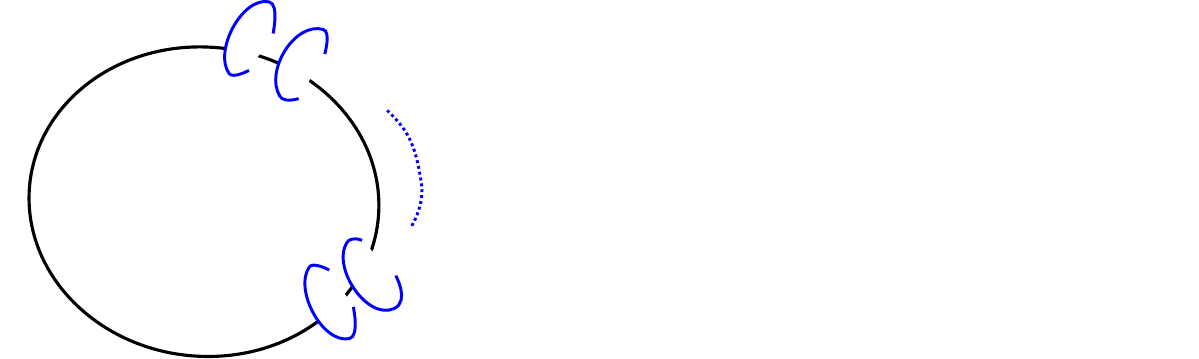}}%
    \put(0.29446815,0.26436177){\color[rgb]{0,0,0}\makebox(0,0)[lt]{\lineheight{1.25}\smash{\begin{tabular}[t]{l}$L_{2d}$\end{tabular}}}}%
    \put(-0.00106927,0.20926874){\color[rgb]{0,0,0}\makebox(0,0)[lt]{\lineheight{1.25}\smash{\begin{tabular}[t]{l}$-4$\end{tabular}}}}%
    \put(0,0){\includegraphics[width=\unitlength,page=2]{L_2d.pdf}}%
    \put(0.90703804,0.26598283){\color[rgb]{0,0,0}\makebox(0,0)[lt]{\lineheight{1.25}\smash{\begin{tabular}[t]{l}$T_{d,d}$\end{tabular}}}}%
    \put(0.61150062,0.21088984){\color[rgb]{0,0,0}\makebox(0,0)[lt]{\lineheight{1.25}\smash{\begin{tabular}[t]{l}$-1$\end{tabular}}}}%
  \end{picture}%
\endgroup%

  \vspace{0.3cm}      
 \caption{\smaller[1]{Left: the $2d$-component link $L_{2d}$ in $L(4,1)$; right: the $d$-component torus link $T_{d,d}$ in $S^3$.}}
 \label{L_2d}
 \end{figure}

\subsection*{Further applications} 
We conclude with a couple of extra applications and some examples. 
First we observe that our results can be applied to the study of the PL genus of knots and links. To make our statement precise we use the following notation: given a link $L$ in $Y$, the boundary of a rational homology four-ball $X$, we define 
\[\tau^X_{\max}(L)= \max \{\tau_\s(L) : \s \in \text{Spin}^c(Y) \text{ extending to } X\} \ .\] 
Similarly, we define the invariant $\tau^X_{\min}(L)$ as the minimum value achieved. 

\begin{teo}
 \label{teo:PL}
 Suppose that $L$ is a link in a rational homology sphere $Y$. Then 
 \[\dfrac{\left|\tau^X_{\max}(L)-\tau^X_{\min}(L)\right|}{2}\leq \mathfrak g^X_{\text{PL}}(L)\ , \]
 where $\mathfrak g^X_{\text{PL}}(L)$ denotes the minimum genus of a $PL$ surface $F$, in the rational homology four-ball $X$, with $|L|$ connected components each bounding a different component of the link $L$. 
\end{teo}

To illustrate Theorem \ref{teo:PL} and  our other results we study two families of quasi-positive links. The first family is shown on the left-hand side of Figure \ref{L_2d}, and we stumbled into it while thinking about the $A_n$-realisation problem \cite{BF}. The second family, depicted in Figure \ref{M_3d}, is a family of quasi-positive links in the lens space $L(9,2)$, and we construct it by playing around with the Nagata transform relating the Hirzebruch surface $F_1=\C P^2\#\overline{\C P^2}$ to the Hirzebruch surface $F_2$. Here is one of our results.

\begin{prop} \label{prop:last}
Let $N_k$ be the link in Figure \ref{N_k}. There exists a rational homology ball Stein filling $(W,J)$ of $L(9,2)$ such that: $\tau_{\max}(N_k)=\frac{k^2+3k}{9}$, and $\tau_{\min}(N_k)=\frac{k^2-3k}{9}$. Consequently, $N_k$ does not bound a pseudo-holomorphic curve in $(W,J)$ for any $k$ that is not divisible by $3$. Furthermore, $N_k$ does not bound a $(-J)$-holomorphic curve for any $k$. 
\end{prop}
The computations for the proof of Proposition \ref{prop:last} are based on an adaptation of the methods of lattice cohomology, a technique that was first explored in \cite{Alfieri}. Note that $N_3=M_3$ and that for this link we are able to construct a $J$-holomorphic curve in $(W,J)$. We conjecture that the link $N_{3d}$ does not bound pseudo-holomorphic curves in $(W,J)$ for $d\geq2$. We believe that it would be interesting to address similar questions for the other families we describe in the paper.

\smallskip
Finally, we mention an application that came up during a conversation with Boyer. The proof is based on some celebrated results by Loi and Piergallini \cite{LP}.

\begin{figure}[t]
\vspace{0.3cm}
 \centering
  \def\svgwidth{7.5cm}
\begingroup%
  \makeatletter%
  \providecommand\color[2][]{%
    \errmessage{(Inkscape) Color is used for the text in Inkscape, but the package 'color.sty' is not loaded}%
    \renewcommand\color[2][]{}%
  }%
  \providecommand\transparent[1]{%
    \errmessage{(Inkscape) Transparency is used (non-zero) for the text in Inkscape, but the package 'transparent.sty' is not loaded}%
    \renewcommand\transparent[1]{}%
  }%
  \providecommand\rotatebox[2]{#2}%
  \newcommand*\fsize{\dimexpr\f@size pt\relax}%
  \newcommand*\lineheight[1]{\fontsize{\fsize}{#1\fsize}\selectfont}%
  \ifx\svgwidth\undefined%
    \setlength{\unitlength}{311.24241626bp}%
    \ifx\svgscale\undefined%
      \relax%
    \else%
      \setlength{\unitlength}{\unitlength * \real{\svgscale}}%
    \fi%
  \else%
    \setlength{\unitlength}{\svgwidth}%
  \fi%
  \global\let\svgwidth\undefined%
  \global\let\svgscale\undefined%
  \makeatother%
  \begin{picture}(1,0.49212209)%
    \lineheight{1}%
    \setlength\tabcolsep{0pt}%
    \put(0,0){\includegraphics[width=\unitlength,page=1]{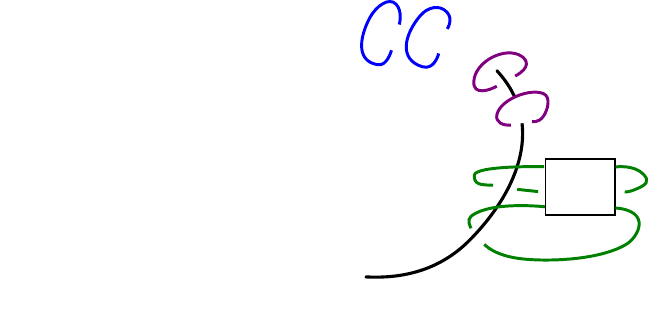}}%
    \put(0.85728917,0.1878178){\color[rgb]{0,0,0}\makebox(0,0)[lt]{\lineheight{1.25}\smash{\begin{tabular}[t]{l}$-1$\end{tabular}}}}%
    \put(0.83218402,0.0156477){\color[rgb]{0,0,0}\makebox(0,0)[lt]{\lineheight{1.25}\smash{\begin{tabular}[t]{l}$M_{3d}$\end{tabular}}}}%
    \put(0,0){\includegraphics[width=\unitlength,page=2]{M_3d.pdf}}%
    \put(0.16153584,0.00611883){\color[rgb]{0,0,0}\makebox(0,0)[lt]{\lineheight{1.25}\smash{\begin{tabular}[t]{l}$-2$\end{tabular}}}}%
    \put(0.58462398,0.00611897){\color[rgb]{0,0,0}\makebox(0,0)[lt]{\lineheight{1.25}\smash{\begin{tabular}[t]{l}$-5$\end{tabular}}}}%
  \end{picture}%
\endgroup%

 \vspace{0.3cm}       
 \caption{\smaller[1]{The link $M_{3d}$ in $L(9,2)$. The $-1$ denotes a negative full-twist. There are $d$ blue, $d$ purple and $d$ green components for a total of $3d$ unframed components.}}
 \label{M_3d}
 \end{figure}

\begin{teo}\label{teo:double} 
 If $K$ is a quasi-alternating quasi-positive knot in $S^3$ then $\tau_{\s_0}(\widetilde{K})=\tau(K)$, where $\widetilde{K}\subset \Sigma(K)$ denotes the fixed point set of the branched double-covering involution, and $\s_0$ is the unique spin structure on $\Sigma(K)$.
\end{teo}
It is indeed conjectured by Grigsby that $\tau_{\s_0}(\widetilde{K})=\tau(K)$ for all alternating knots in $S^3$, and so far this was only confirmed for alternating torus knots $T_{2,2n+1}$ in \cite{ACS}, and by unpublished computer experiments performed by Celoria.  

\subsection*{Acknowledgements} {\smaller[1] We would like to thank Steven Boyer for some suggestions that led to the application to double-covers of quasi-alternating links, Daniele Celoria for letting us use some of his computer programs for numerical experiments, Ian Zemke for sending us some corrections, and the anonymous referee for the many useful suggestions and for pointing out a mistake in a previous version of Theorem 1.7. We thank Peter Feller for inspiring this work with a wonderful talk about the $A_n$-realisation problem, and Burak \"Ozba\u gc\i\:for carefully reviewing an earlier version of this manuscript.
Last but not least, we thank Irving Dai and Irena Matkovi\v c for generously sharing their expertise. The ideas for the paper were developed at CIRGET, based at the Universit\'e du Qu\'ebec \`a Montr\'eal (UQAM).}

\section{The different notions of tau-invariant} \label{section:two}
\subsection{Floer type chain complexes}We start with  some homological algebra.
\label{subsection:floer}
\begin{defin}
A Floer type chain complex $C^{-}$ is a finitely generated complex over the ring of polynomials $\F[U]$ such that:
\begin{itemize}
\item $C^{-}$ is free, that is 
\[\left( C^{-}=\bigoplus_{ \x \in \textbf B }  \ \F[U] \cdot \x    ,  \  \partial \x = \sum_{\y} c_{\x\y} \ U^{m_{\x\y}} \cdot  \y  \right)  \]
for some finite basis $\textbf B$, and  non-negative exponents $m_{\x\y}$.
\item the complex is graded so that multiplication by $U$ drops the grading by two,
\item if we set $C^{\infty}=C^{-}\otimes_{\F[U]} \F[U, U^{-1}]$ to be the localised chain complex then there is an isomorphism of graded modules
\[H_*(C^{\infty})\simeq  \F[U, U^{-1}] \otimes_{\F} \left(\F_{(-1)}\oplus \F_{(0)}\right)^{\otimes \ell-1}\simeq  \F[U,U^{-1}]^{\oplus 2^{\ell-1}} \:.\] This is the same as asking $H_*(C^-)$ to have rank $2^{\ell-1}$ as an $\F[U]$-module.
 \end{itemize}
\end{defin}

Note that $\left(\F_{(-1)}\oplus \F_{(0)}\right)^{\otimes \ell-1}$ has $\binom{\ell}{k}$ generators of grading $\text{gr}=-k$ for every $0\leq k\leq \ell-1$. In particular, it has a unique generator $\e_{1-\ell}= (1, 0) \otimes \dots \otimes (1,0)$ in degree $\text{gr}=-(\ell-1)=1-\ell$, and a unique generator $\e_0=(0, 1) \otimes \dots \otimes (0, 1) $ in grading zero.

For a Floer type chain complex $C^{-}$ we define the correction term $d(C^{-})$ to be the maximum grading of a homogeneous cycle $z$ whose homology class is not $U$-torsion in $H_*(C^-)$. We will also consider the associated hat complex
\[\widehat{C}= C^{-}\otimes_{ \F[U]}\F[U]/U =C^-/U\]
 that is obtained from $C^{-}$ substituting in $U=0$, that is:
 \[\widehat{C}=\left( \bigoplus_{ \x \in \textbf B }  \ \F \cdot \x    ,  \  \partial \x = \sum_{\y} \widehat{c}_{\x\y} \cdot  \y  \right) \ ,  \]
 where $ \widehat{c}_{\x\y}=1$ if $c_{\x\y}=1$ and $m_{\x\y}=0$, and zero otherwise. We shall regard the hat complex $\widehat{C}$ as a chain complex over  the ground field $\F$, and its homology $H_*(\widehat{C})$  as a finite dimensional graded vector space.

\label{subsection:theta}
Suppose that $C^-$ is a Floer type chain complex with correction term $d=d(C^-)$. Define $ C^+= C^\infty/UC^-$ and consider the short exact sequence
\begin{center}
\begin{tikzcd}
 0 \arrow[r] & C^- \arrow[r, "i\:\circ\:U"] & C^\infty \arrow[r, "\pi"] &  C^+ \arrow[r] & 0 \:.
\end{tikzcd} 
\end{center} 
This induces a long exact sequence 
\begin{center}
\begin{tikzcd}
 \dots \arrow[r] & H_*( C^-) \arrow[r, "i_*\circ\:U"] & H_*(C^\infty) \arrow[r, "\pi_*"] &  H_*(C^+) \arrow[r] & \dots 
\end{tikzcd} 
\end{center} 
from where one concludes that $H_*(C^+)$ contains a copy of \[\F[U, U^{-1}]/U \F[U] \otimes_{\F} \left(\F_{(-1)}\oplus \F_{(0)}\right)^{\otimes \ell-1}\] as a direct summand. Note that $\Xi^+:=1\otimes \e_{1-\ell}$ and $\Theta^+:=1\otimes \e_0$ give two preferential elements of $H_*(C^+)$ in grading $d+1-\ell$ and $d$ respectively. 
We denote by $H^+_{\Xi}$ and $H^+_{\Theta} \subset H_*(C^+)$ the $1$-dimensional $\F$-vector subspaces spanned by $\Xi^+$ and $\Theta^+$.

Multiplication by $U$ induces a short exact sequence 
\begin{center}
\begin{tikzcd}
 0 \arrow[r] & \widehat C \arrow[r, "\rho"] & C^+ \arrow[r, "U"] &  C^+ \arrow[r] & 0 
\end{tikzcd} 
\end{center} 
that gives in return a long exact sequence
\begin{center}
\begin{tikzcd}
 \dots \arrow[r,"\delta"] & H_*(\widehat C) \arrow[r, "\rho_*"] & H_*(C^+) \arrow[r, "U"] &  H_*(C^+) \arrow[r,"\delta"] & \dots \:.
\end{tikzcd} 
\end{center} 
Consequently, one has that $\rho_*(H_*(\widehat C))=\Ker U$, and since $U\cdot\Xi^+=U \cdot\Theta^+=0$ we can find some (non-canonical) homology classes in $H_*(\widehat C)$ which are mapped to $\Xi^+$ and $\Theta^+$. We can characterise those classes by studying the long exact sequence
\begin{center}
\begin{tikzcd}
 \dots \arrow[r,"\delta'"] & H_*(C^-) \arrow[r, "U"] & H_*(C^-) \arrow[r, "\psi_*"] & H_*(\widehat C) \arrow[r,"\delta'"] & \dots \:,
\end{tikzcd} 
\end{center} 
which allows us to prove the following lemma which is well known among Heegaard Floer homology experts.
\begin{lemma}
 \label{lemma:newer}
 Let us consider a homology class $\alpha\in H_*(\widehat C)$ of degree $d$ (resp. $d+1-\ell$). Then
 \begin{itemize}
     \item $\rho_*(\alpha)\subset H^+_\Theta$ (resp. $\rho_*(\alpha)\subset H^+_{\Xi}$) if and only if $\alpha\in\Imm \psi_*$;
     \item $\rho_*(\alpha)=\Theta^+$ (resp. $\rho_*(\alpha)=\Xi^+$) if and only if $\alpha=\psi_*(\delta)$ and $\delta$ is a non-torsion generator of the corresponding $\F[U]$-tower of $H_*(C^-)$;
     \item $\alpha\notin\Imm \psi_*$ if and only if $\rho_*(\alpha)\notin U^n\cdot H_*(C^+)$ for some $n\geq1$.
 \end{itemize}
\end{lemma}
\begin{proof}
 Combining the previous long exact sequences yields the following commutative diagram.
 \[\begin{tikzcd}
   H_*(C^-) \arrow[r, "i_*"] \arrow[d, "\psi_*"] & H_*(C^\infty) \arrow[d, "\pi_*"] \\
   H_*(\widehat C) \arrow[r, "\rho_*"] & H_*(C^+) 
 \end{tikzcd}\]
 The second item is proved in \cite[Lemma 2.3]{Raoux}. Denote by $\mathcal K$ the subspace of $H_*(C^-)$ generated by $U\cdot H_*(C^-)$ and $\text{Tor }H_*(C^-)$, we start by showing that $\alpha\in\Ker \rho_*$ if and only if there is a class $\beta\in\mathcal K$ such that $\psi_*(\beta)=\alpha$. If $\beta\in\mathcal K$, and $\psi_*(\beta)=\alpha$, then $0=\pi_*(i_*(\beta))=\rho_*(\psi_*(\beta))=\rho_*(\alpha)$. Conversely, if $\rho_*(\alpha)=0$ then there is a $c\in C^\infty$ such that $\partial c=a+Ub$, where $[a]=\alpha$ in $H_*(\widehat C)$ and $b\in C^-$; then $\partial(a+Ub)=0$ which means $a+Ub$ is a cycle in $C^-$ and $\psi_*[a+Ub]=\alpha$. The class $\beta=[a+Ub]$ is in $\mathcal K$ because otherwise $\pi_*(i_*(\beta))\neq0$.

 We continue by proving that if $\alpha\notin\Imm \psi_*$ then $\rho_*(\alpha)\notin U^n\cdot H_*(C^+)$ for some $n\geq1$, where we immediately note that the latter condition is equivalent to say $\rho_*(\alpha)\notin\Imm \pi_*$. By contradiction, we would find a $\beta\in H_*(C^-)$ such that $\rho_*(\alpha)=\pi_*(i_*(\beta))=\rho_*(\psi_*(\beta))$, and then $\psi_*(\beta)+\alpha$ belongs to $\Ker \rho_*$. For what we said in the previous paragraph there exists another class $\gamma\in H_*(C^-)$ such that $\psi_*(\beta+\gamma)=\alpha$, which goes against the initial assumption.
 This completes the proof: the first item and the other half of the third one follow easily by the previous discussion. 
\end{proof}
In what follows for a Floer type chain complex $C^-$ we denote $D^-:=(C^+)^\bullet$ and $D^+:=(C^-)^\bullet$ the complexes dual to $C^+$ and $C^-$ with respect to the basis $\textbf B$.
By taking the dual of the short exact sequences we previously stated we can then define $\widehat{D}:=\widehat C^\bullet$ and the maps $\rho^*:=\psi_*^\bullet:H_*(\widehat{D})\rightarrow H_*(D^+)$ and $\psi^*:=\rho_*^\bullet:H_*(D^-)\rightarrow H_*(\widehat{D})$. 

We have isomorphisms $H_i(\widehat D)^\bullet\cong H_{-i+1-\ell}(\widehat C)$ and $H_i(D^-)^\bullet\cong H_{-i+1-\ell}(C^+)$. Under the identification of a finitely generated vector space with its bidual, we have the canonical duality bilinear form \[\langle\cdot\:,\:\cdot\rangle:H_i(\widehat C)\otimes H_{-i+1-\ell}(\widehat D)\longrightarrow\F\:.\]
Therefore, a subspace $V\subset H_*(\widehat C)$ is canonically identified with $W^\perp$ for some subspace $W\subset H_*(\widehat D)$ and we just write $V=W^\perp$.

\subsection{Alexander type filtrations}
We will be interested in some specific type of filtrations on Floer type chain complexes.
\begin{defin} Let $C^{-}$ be a Floer type chain complex. An Alexander type filtration on $C^{-}$ is a filtration $\mathcal{F}_{*}: \dots \subset  \mathcal{F}_{n-1} \subset \mathcal{F}_{n} \subset \dots$ such that:
 \begin{itemize}
 \item the differential preserves the filtration, that is $\partial(\mathcal{F}_{n} ) \subset \mathcal{F}_{n}$ for all $n\in \Z$;
 \item multiplication by $U$ drops the filtration level by one, that is $U \cdot \mathcal{F}_{n}\subset \mathcal{F}_{n-1}$.
\end{itemize}  
 A filtered Floer type chain complex is a pair $(C^{-}, \mathcal{F})$ for some Alexander type filtration $\mathcal{F}_*$.
\end{defin} 

An Alexander type filtration $\mathcal{F}_*$ induces a filtration on the associated hat complex $\widehat{C}$ by taking the projection of $\mathcal F_*$ through $\psi:C^-\rightarrow\widehat C$, and on the dual hat complex $\widehat D$ by taking $(\mathcal F_{-*-1})^\perp$.
One can fix a non-zero class $\alpha \in H_*(\widehat{C})$ and define the $\tau_\alpha$-invariant of an Alexander type filtration $\mathcal{F}_*$ as
\[\tau_\alpha(\mathcal{F})= \min \big\{ \ m \in \Z \ : \ \alpha\in H_*(\mathcal{F}_{m})\big\} \ .\]
Similarly, we define the invariant $\tau_\text{top}(C^{-}, \mathcal{F})$ of the filtration $\mathcal{F}_*$ to be the minimum filtration level containing a cycle which represents a class in the preimage of $\Theta^+$:
\[\tau_\text{top}(C^{-}, \mathcal{F})= \min \big\{ \ m \in \Z \ : \ H_*(\mathcal{F}_{m})\text{ contains an } \alpha \text{ such that }\rho_*(\alpha)=\Theta^+\big\}\:,\]
and the invariant $\tau_\text{bot}(C^{-}, \mathcal{F}_*)$ as follows:
\[\tau_\text{bot}(C^{-}, \mathcal{F})= \min \big\{ \ m \in \Z \ : \ H_*(\mathcal{F}_{m})\text{ contains a } \beta \text{ such that }\rho_*(\beta)=\Xi^+ \big\}\:.\]
We show some generalisations of the symmetries under orientation reversing in \cite{OSlinks,Raoux}.
\begin{lemma}
 \label{lemma:mirror}
 Denote by $D^-$ the dual of $C^+$ and by $-\mathcal F$ the filtration on $\widehat D=\widehat C^\bullet$. We have that \[\tau_\text{bot}(D^-,-\mathcal{F})=-\tau_\text{top}(C^{-},\mathcal F)\:.\] Furthermore, if $\alpha\in H_*(\widehat C)$ is a non-zero class then \[\tau_\alpha(\mathcal F)=-\min\big\{\tau_\gamma(-\mathcal F)\: : \: \gamma\in H_{-*+1-\ell}(\widehat D)\text{ and }\gamma\notin\alpha^\perp\big\}\:.\] 
\end{lemma}
\begin{proof} 
 Let us write $-\mathcal F_n:=\mathcal F_{-n-1}^\perp$ for the $n$-th level of the filtration on $\widehat D$. We adapt the proof from \cite[Proposition 3.10]{Raoux}; in there the following equality is shown in the knot case, but it holds verbatim for links: 
 \begin{equation}
  H_{-*+1-\ell}(-\mathcal F_{-n-1})=H_{-*+1-\ell}(\mathcal F_n^\perp)=H_*(\mathcal F_n)^\perp\:.
  \label{eq:dual}
 \end{equation} 
 Suppose that $m<\tau_\text{top}(C^{-}, \mathcal{F})$. By definition we have that 
 \begin{equation}
  H_d(\mathcal F_m)\cap \rho_d^{-1}(\Theta^+)=\emptyset\:;
  \label{eq:empty}
 \end{equation}
 moreover, since $\rho_*^{-1}(H^+_{\Theta})=\Imm \psi_*$ by the first item of Lemma \ref{lemma:newer}, one has 
 \begin{equation}
  \rho_*^{-1}(\Theta^+)^\perp=\rho_*^{-1}(H^+_{\Theta})^\perp=(\Imm \psi_*)^\perp=\Ker \psi_*^\bullet=\Ker \rho^{-*+1-\ell}\:.
  \label{eq:theta}
\end{equation}
 Combining Equation \eqref{eq:dual} with the dual of Equation \eqref{eq:empty} yields \[\text{Span}(H_{-d+1-\ell}(-\mathcal F_{-m-1}),\Ker \rho^{-d+1-\ell})=H_{-d+1-\ell}(\widehat D)\:,\] and then a given $\gamma$ such that $\rho^*(\gamma)=\Xi^+$ can be written as $\gamma=\alpha+\alpha_0$, where $\alpha_0\in\Ker \rho^*$ and $\alpha\in H_{-d+1-\ell}(-\mathcal F_{-m-1})$, and consequently $\rho^{-d+1-\ell}(\alpha)=\Xi^+$. If we set $n=-m-1$ then such an $\alpha$ exists for every $n\geq-\tau_\text{top}(C^{-}, \mathcal{F})$, which implies $\tau_\text{bot}(D^{-}, -\mathcal{F})\leq-\tau_\text{top}(C^{-}, \mathcal{F})$.

 Now suppose that $\tau_\text{top}(C^{-}, \mathcal{F})\leq m$. We now have that there is an $\alpha$ such that $\Theta^+=\rho_d(\alpha)$ inside $H_d(\mathcal F_m)$, which means $H_{-d+1-\ell}(-\mathcal F_{-m-1})\subset\alpha^\perp$; therefore, for every $n=-m-1<-\tau_\text{top}(C^{-}, \mathcal{F})$ each class $\gamma$ in $H_{-d+1-\ell}(-\mathcal F_{n})$ satisfies $\langle\alpha,\gamma\rangle=0$. If $\rho^*(\gamma)=\Xi^+$ then $\gamma\in \Ker \rho_*^\perp$ from the dual of Equation \eqref{eq:theta}, and then by the first two items of Lemma \ref{lemma:newer} one has
 \[\gamma\in\text{Span}(\alpha,\Ker \rho_*)^\perp=\rho_*^{-1}(H^+_{\Theta})^\perp=\Ker \rho^*\] which is a contradiction. This implies that no $\gamma$ in $H_{-d+1-\ell}(-\mathcal F_{n})$ is mapped to $\Xi^+$ and then $-\tau_\text{top}(C^{-}, \mathcal{F})\leq\tau_\text{bot}(D^{-}, -\mathcal{F})$.
 
 For the second equality, assume that $m<\tau_\alpha(\mathcal F)$. By definition we have that $\alpha\notin H_*(\mathcal F_m)$, which means $H_{-*+1-\ell}(-\mathcal F_{-m-1})\not\subset\alpha^\perp$ by duality; hence, there exists a class $\gamma\in H_{-*+1-\ell}(-\mathcal F_{-m-1})$ such that $\gamma\notin\alpha^\perp$. If we set $n=-m-1$ then the latter condition holds for every $n\geq-\tau_\alpha(\mathcal F)$, which implies $\tau_\gamma(-\mathcal F)\leq-\tau_\alpha(\mathcal F)$.

 Now we take $\tau_\alpha(\mathcal F)\leq m$. We now have that $\alpha\in H_*(\mathcal F_m)$, which means $H_{-*+1-\ell}(-\mathcal F_{-m-1}) \\\subset\alpha^\perp$ again by duality; therefore, each class in $H_{-*+1-\ell}(-\mathcal F_{n})$ is contained in $\alpha^\perp$ for every $n=-m-1<-\tau_\alpha(\mathcal F)$, which implies $-\tau_\alpha(\mathcal F)\leq\tau_\gamma(-\mathcal F)$ for every $\gamma\notin\alpha^\perp$.
\end{proof}

\subsection{Heegaard Floer homology and contact invariants} 
Heegaard Floer homology associates to a $\Spin^c$-rational homology sphere $(Y, \mathfrak{s})$, represented by a multi-pointed Heegaard diagram, a Floer type chain complex $CF^-(Y, \mathfrak{s};\ell)$ with
\[HF^\infty(Y, \s;\ell)= H_*(CF^-(Y, \mathfrak{s};\ell)\otimes_{\F[U]} \F[U,U^{-1}])\simeq\F[U, U^{-1}] \otimes_{\F} \left(\F_{(-1)}\oplus \F_{(0)}\right)^{\otimes \ell-1}  \]
where $\ell$ denotes the number of basepoints. 
To avoid confusion with the Ozsv\'ath-Szab\'o notation, we write $\widehat{HF}(Y,\s)$ only when the group is constructed using a Heegaard diagram of $Y$ having a unique basepoint. When there are $\ell$ basepoints, we non-canonically write
\[\widehat{HF}(Y,\s;\ell) \simeq \widehat{HF}(Y,\s)\otimes_{\F} \left(\F_{(-1)}\oplus \F_{(0)}\right)^{\otimes \ell-1}\:.\]
If $L$ is a link in $Y$, represented with an $\ell$-pointed Heegaard diagram by the additional choice of extra $\ell$ basepoints, then the chain group $CF^-(Y, \mathfrak{s};\ell)$ can be equipped with an Alexander type filtration $\mathcal{F}^L$ \cite{OSlinks}. This gives rise to the invariants
\[\tau(Y,L,\mathfrak{s})=\tau_\text{top}(CF^-(Y, \mathfrak{s};\ell),\mathcal{F}^L)\hspace{1cm}\text{ and }\hspace{1cm}\tau^*(Y,L,\mathfrak{s})=\tau_\text{bot}(CF^-(Y, \mathfrak{s};\ell),\mathcal{F}^L)\] 
associated to the $\Spin^c$-structure $\mathfrak{s}$. Alternatively, one can choose a non-zero class $\alpha \in \widehat{HF}_*(Y, \mathfrak{s};\ell)$ and look at the invariant \[\tau_\alpha(L) = \tau_{\alpha}(\mathcal{F}^L)\:.\]
Of course in some cases the two invariants are related to each other.
\begin{remark}
 If $Y$ is an $L$-space then $\widehat{HF}(Y, \mathfrak{s})$ is one dimensional, and we can call $\theta$ and $\theta^*$ the only non-zero classes in $\widehat{HF}(Y,\s;\ell)$ of degree $d$ and $d+1-\ell$ respectively. Therefore, one has \[\tau_\theta(L) =\tau(Y,L,\s)\hspace{1cm}\text{ and }\hspace{1cm}\tau_{\theta^*}(L) =\tau^*(Y,L,\s)\:.\]
\end{remark}
In \cite{GRS,Raoux} the invariant $\tau$ is studied in the context of knots sitting on the boundary $(Y,\s)$ of a rational homology four-ball $X$. The authors prove the following:
\begin{enumerate}
 \item if $F\hookrightarrow X$ is a compact, connected and oriented surface with $\partial F=K$, and $\s$ extends to $X$, then $|\tau(Y,K,\s)|\leq g(F)$;
 \item $\tau(-Y,K,\s)=-\tau(Y,K,\s)$;
 \item if $K_i$ are knots in $(Y_i,\s_i)=\partial(X_i,\mathfrak u_i)$ such that $\mathfrak u_i\lvert_{Y_i}=\s_i$ for $i=1,2$ then \[\tau(Y_1\#Y_2,K_1\#K_2,\s_1\#\s_2)=\tau(Y_1,K_1,\s_1)+\tau(Y_2,K_2,\s_2)\:.\]
\end{enumerate}
In Section \ref{section:six} we prove, with the due modifications, a version of Property 1 in the case of links, while Lemma \ref{lemma:mirror} provides the one for Property 2. 

If $(M,\xi)$ is a contact three-manifold, then the hyperplane distribution $\xi$ has an associated $\Spin^c$-structure $\mathfrak{s}_\xi$, and a homogeneous cycle $x(\xi) \in \widehat{CF}(-M,\mathfrak{s}_\xi)$, whose homology class $\widehat c(\xi)$ is called the contact invariant of $(M,\xi)$ \cite{OSz-contact}. In \cite{Cavallo-L} it is explained how to define the invariant $\widehat c(\xi)$ when the Heegaard diagram of $-M$ is multi-pointed; in particular, from \cite[Lemma 5.2]{Cavallo-L} and \cite[Theorem 2.8]{Cavallo-upsilon} it follows that 
its Maslov grading is equal to $-d_3(\xi)+1-\ell$, where
$d_3(\xi)$ denotes the Hopf invariant of the plane field $\xi$. 
When $\widehat c(\xi)$ is non-vanishing we can define a link invariant by setting $\tau_\xi(L):=-\tau_{\widehat c(\xi)}(L)$: this is the invariant studied by Hedden \cite{Hedden}. We show that if $(M,\xi)$ also admits a Stein filling then $\widehat c(\xi)$ satisfies a stronger property. 
\begin{teo}\label{teo:supported}
 If a contact three-manifold $(M,\xi)$ admits a strong symplectic filling 
 then $\widehat c(\xi)\notin\theta^\perp$, where $\theta=\widehat F_{W,J}(\textbf e_0)$ is the image of the cobordism map induced by $W$.    
\end{teo}
\begin{proof} 
 We consider the map \[\widehat F_{\overline W,J}:\widehat{HF}_{-d_3(\xi)+1-\ell}(-M, \s_\xi;\ell)\longrightarrow\widehat{HF}_{1-\ell}(S^3;\ell)\] where $\widehat F_{\overline W,J}$
 is the cobordism map induced by the strong symplectic filling $(W,\omega)$ turned upside-down, and $J$ is the $\Spin^c$-structure induced by $\omega$.
 The gradings satisfy $\deg(\widehat F_{W,J})=\deg(\widehat F_{\overline W,J})=d_3(\xi)$.
 Of course, being the contact class defined by an intersection point, it is homogeneous.   

 Indeed, we can conclude that the contact class $ \widehat c(\xi) $ cannot be in the kernel of 
 $\widehat F_{\overline W,J}$ since $\widehat{F}_{\overline W,J}( \widehat c(\xi) )=\textbf e_{1-\ell}$ from \cite[Remark 2.14]{Ghiggini2}. 
 It is then enough to show that $\Ker \widehat F_{\overline W,J}=\theta^\perp$. 
 Since $\widehat F_{\overline W,J}$ is the dual of $\widehat F_{W,J}$ this follows from $\Ker \widehat F_{\overline W,J}=(\Imm \widehat F_{W,J})^\perp=\theta^\perp$.    
\end{proof}  
This obviously holds when $c^+(\xi):=\rho_*(\widehat c(\xi))$ coincides with $\Xi^+$ in the group $HF^+(-M,\s_\xi)$. This happens for example
when $(M,\xi)$ is supported by a planar open book \cite{OSSz}, or when it is the boundary of a rationally convex domain in $\C^2$ \cite{MT}. We immediately have the following corollary.  
\begin{cor}
 \label{cor:supported}
 Suppose that $(M,\xi)$ has a strong symplectic filling, and that $L\subset M$ is an $\ell$-component link, then $\tau_\xi(L)\leq\tau_\theta(L)$ where $\theta\in\widehat{HF}(Y,\s;\ell)$ is the image of $\widehat F_{W,J}$.
\end{cor}
\begin{proof}
 In Theorem \ref{teo:supported} we have proved that $\widehat c(\xi)$ is not in $\Ker\widehat F_{\overline W,J}=\theta^\perp$, and using Lemma \ref{lemma:mirror} we obtain $-\tau_\theta(L)=\min\big\{\tau_\gamma(L)\: : \: \gamma\in\widehat{HF}_{-d_3(\xi)+1-\ell}(-M,\s_\xi;\ell)\text{ and }\gamma\notin\theta^\perp\big\}\leq\tau_{\widehat c(\xi)}(L)$. We can now reverse the inequality and write $\tau_\xi(L)=-\tau_{\widehat c(\xi)}(L)\leq\tau_\theta(L)$. 
\end{proof}

\section{Quasi-positive links in Stein fillable three-manifolds}\label{section:three}
In what follows $Y$ denotes a closed three-manifold, and $(B\subset Y,\pi:Y\setminus B \to S^1)$ an open book for $Y$. For basic definitions and results regarding open books 
we refer to \cite{OS2}.

\begin{defin}[Braids in an open book]
A link $L$ in $Y$ is braided around an open book $(B, \pi)$ of a three-manifold $Y$ if $L$ intersects transversely every page $F_\theta= \pi^{-1}(\theta)$ of the open book in a fixed number $n>0$ of distinct points. 
\end{defin}

\noindent Note that a link braided around an open book is naturally transverse to the contact structure supported by the given open book. We shall describe braids through (abstract) pointed open books as in \cite{Hayden}. 	

\subsection{Pointed open books}
Let $(S,\phi)$ be an abstract open book for $(B,\pi)$. In what follows we fix a set of marked points $P=\{p_1,...,p_n\}$ in the interior of $S$, and we assume that the monodromy $\phi:S\to S$ preserves \emph{point-wise} both the boundary $\partial S$, and the marked set $P$. 

We denote by $B_n(S,P)$ the set of all diffeomorphisms $S\to S$ that preserve $P$ \emph{set-wise}, $\partial S$ point-wise, and are isotopic to the identity. Note that for any choice of $\sigma \in B_n(S,P)$ the pair $(S, \sigma\circ \phi)$ defines the same open book $(B,\pi)$ as $(S,\phi)$. However, the choice of $\sigma$ gives rise to a transverse link $T$ in $Y$ braided around $(B,\pi)$. This is defined by the suspension $T=P\times I/\beta$ of the marked set $P$, where $\beta=\sigma\circ \phi$.  We call the four-tuple $(S, P, \phi, \sigma)$ a \emph{pointed open book}. We call  $\beta= \sigma \circ \phi$ the braid encoded by the pointed open book $(S, P, \phi, \sigma)$, and $\hat{\beta}=T$ the closure of $\beta$.

We know \cite{Geiges,Pavelescu} that any transverse link $L$ braided around an open book $(B,\pi)$ is encoded by a pointed open book $(S, P, \phi, \sigma)$. Furthermore, any transverse link in the contact three-manifold supported by $(B,\pi)$ is isotopic to a braid as in the previous paragraph. There is also a generalisation of Markov's theorem \cite[Theorem 2.3]{Hayden}: two braids in a contact three-manifold represent the same transverse link if and only if they differ by conjugation of the pointed open book, positive stabilisation of the braid, and positive Hopf plumbing of the open book. 

\begin{defin}[Quasi-positive braids in arbitrary contact manifolds]
A braid in an open book is quasi-positive if it can be encoded by an abstract pointed open book whose pointed monodromy is isotopic to a product of positive half-twists and arbitrary Dehn twists. 
\end{defin}

In what follows we will study quasi-positive braids in open books associated to Stein fillable structures. These manifest themselves as the boundary of pseudo-holomorphic curves embedded in the fillings.

\subsection{Holomorphic curves in Stein fillings}
It is known that every Stein fillable contact three-manifold $(M,\xi)$ admits a \emph{positive allowable open book}, i.e. an abstract open book $(S,\phi)$ where the monodromy $\phi$ factors as a product of positive Dehn twists \cite{Giroux}. The following was recently proved in \cite{Hayden}.

\begin{teo}[Hayden]
 \label{teo:Hayden}
 The boundary of any properly embedded pseudo-holomorphic curve in a Stein domain $(W,J)$ is transversely isotopic to a quasi-positive braid, encoded by a pointed and positive allowable open book.   
\end{teo}

Hayden's proof is based on the following observations (we review them because they are relevant to our discussion, and guide the ideas in our proof).
Given a pseudo-holomorphic curve $C$ in a Stein domain $(W,J)$, and a smooth convex hyper-surface $M\subset W$, consider the intersection $C\cap M$. If this intersection is transverse, then $C\cap M$ is a transverse link in $(M,\xi)$, where $\xi$ is the contact structure induced naturally on $M$ by the Stein structure. Indeed, since $TC$ and $\xi$ consist of complex lines in $TW$, their intersection at any point has real dimension zero or two. It follows that $C\cap M$ is transverse to the contact structure except
for those points $p$ at which $C$ is tangent to $M$, where $T_pC=\xi_p\subset T_pM$.

Denote by $\rho:W\rightarrow[0,\infty)$ a $J$-convex function on $(W,J)$.
We say that a smooth, oriented surface $F\subset W$ is \emph{ascending} if $F$ contains no critical points of $\rho$, the restriction $\rho\lvert_F$ is a Morse function, and, except
at its critical points, each level set $(\rho\lvert_F)^{-1}(c)$ is positively transverse to the contact structure on $\rho^{-1}(c)$.

We know from \cite[Proposition 4.10]{Hayden} that every critical point $p$ of an ascending surface $F$ is such that $T_pF$ is a complex line in $T_pW$.
Every such critical point can then be described as positive or negative according to whether the intrinsic orientation on $T_pF$ agrees or not, respectively, with the complex orientation. In general, we can further classify these points as elliptic or hyperbolic; see \cite[Definition 8.3.3]{OS2} for the general definition.

It is then stated in \cite[Proposition 4.11]{Hayden} that, with respect to a generic $J$-convex function $\rho$, a pseudo-holomorphic curve in a Stein surface $(W,J)$ is ascending with only positive critical points. Such a surface is called \emph{positive ascending surface}. In his paper Hayden proves Theorem \ref{teo:Hayden} above as a special case of \cite[Theorem 4.2]{Hayden}, where it is shown that the boundary of any positive ascending surface, properly embedded in a Stein domain $W$, is transversely isotopic to a quasi-positive braid encoded by a (positive allowable) pointed open book for $\partial W$.

\subsection{Positive Bennequin surfaces}
\label{subsection:Bennequin}
\begin{figure}[t]
\vspace{0.6cm}
 \centering
 \def\svgwidth{13cm}
\begingroup%
  \makeatletter%
  \providecommand\color[2][]{%
    \errmessage{(Inkscape) Color is used for the text in Inkscape, but the package 'color.sty' is not loaded}%
    \renewcommand\color[2][]{}%
  }%
  \providecommand\transparent[1]{%
    \errmessage{(Inkscape) Transparency is used (non-zero) for the text in Inkscape, but the package 'transparent.sty' is not loaded}%
    \renewcommand\transparent[1]{}%
  }%
  \providecommand\rotatebox[2]{#2}%
  \newcommand*\fsize{\dimexpr\f@size pt\relax}%
  \newcommand*\lineheight[1]{\fontsize{\fsize}{#1\fsize}\selectfont}%
  \ifx\svgwidth\undefined%
    \setlength{\unitlength}{579.85716248bp}%
    \ifx\svgscale\undefined%
      \relax%
    \else%
      \setlength{\unitlength}{\unitlength * \real{\svgscale}}%
    \fi%
  \else%
    \setlength{\unitlength}{\svgwidth}%
  \fi%
  \global\let\svgwidth\undefined%
  \global\let\svgscale\undefined%
  \makeatother%
  \begin{picture}(1,0.33333333)%
    \lineheight{1}%
    \setlength\tabcolsep{0pt}%
    \put(0,0){\includegraphics[width=\unitlength,page=1]{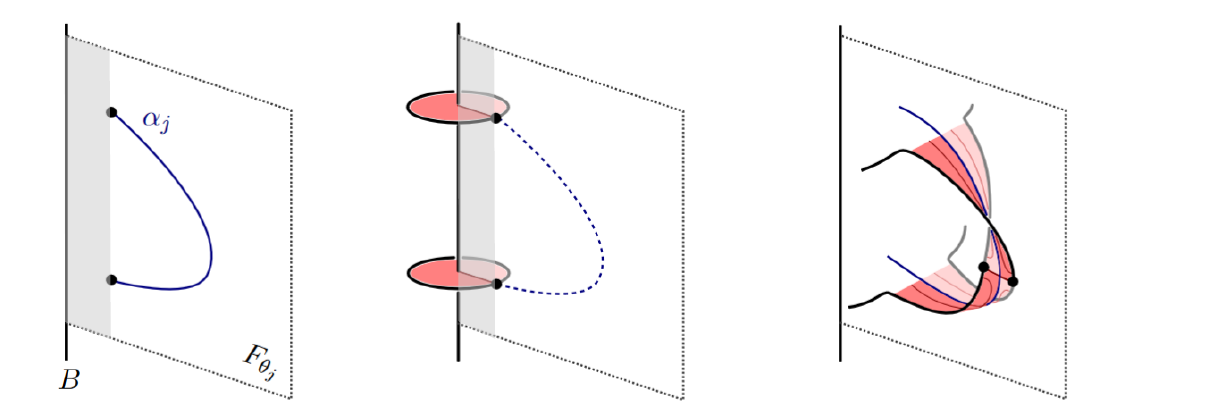}}%
    \put(0.06434867,0.31134127){\color[rgb]{0,0,0}\makebox(0,0)[lt]{\lineheight{1.25}\smash{\begin{tabular}[t]{l}$\Omega$\end{tabular}}}}%
    \put(0.38820711,0.31316534){\color[rgb]{0,0,0}\makebox(0,0)[lt]{\lineheight{1.25}\smash{\begin{tabular}[t]{l}$\Omega$\end{tabular}}}}%
  \end{picture}%
\endgroup%

 \vspace{0.3cm}       
 \caption{\smaller[1]{Attaching positively twisting bands in a Bennequin surface. The picture is taken from \cite[Figure 5]{Hayden}.}}
 \label{Bennequin}
\end{figure}

Suppose that a transverse link $T$ can be expressed as the closure of a  quasi-positive braid  in a contact three-manifold $(M,\xi)$. This means that we can find a pointed open book $(S,P,\sigma,\phi)$ such that 
\[\sigma=H_{\alpha_1}\circ\cdots\circ H_{\alpha_m}\ , \] 
where $H_{\alpha_1}, \dots, H_{\alpha_m}$ denote positive half-twists along embedded arcs $\alpha_1, \dots, \alpha_m \subset S$ with endpoints in $P$. For the following discussion we fix a collar neighbourhood $\Omega$ of $\partial S$, and we draw the marked points $P=\{p_1, \dots, p_k\}$ on the interior boundary $\partial (\Omega \setminus \partial S)$ of said collar neighbourhood. Under the additional assumption that: 
\begin{center}
(*) \emph{  the open book monodromy $\phi$ fixes point-wise the collar neighbourhood $\Omega$ of $\partial S$ }
\end{center}
one can construct an immersed surface $\mathcal{B}_T \subset M$ with ribbon singularities bounding $T$. This is a so called \emph{Bennequin surface} we learned about from \cite{Hayden}. It is constructed via the following algorithm: 
\begin{itemize}
\item\emph{Input:} pointed open book $(S,P,\sigma,\phi)$ satisfying assumption (*) above, and a factorisation $\sigma=H_{\alpha_1}\circ\cdots\circ H_{\alpha_m}$ in positive half-twists.
\item\emph{Step 2:} Fix a set of disjoint embedded arcs $c_1, \dots, c_n \subset \Omega$ joining the marked points $p_1, \dots, p_n$  to some points $q_1, \dots , q_n \in \partial S$. Let $\mathcal B_0\subset M$ be the disjoint union of disks swept out by $c_1, \dots, c_n$ in the open book, that is $\mathcal{B}_0= \bigcup_{i=1}^n c_i\times I/\sim$. 
\item\emph{Step 3:} Attach bands $h_j$ to $\mathcal B_0$ as follows: the core of $h_j$ is a copy of the arc $\alpha_j$, lying in the
page $S_{\theta_j}$ for $\theta_j=\frac{2\pi j}{m+1}$. We can extend this core to a positively twisted band $h_j$ as suggested by Figure \ref{Bennequin}.
\item\emph{Output:} the immersed Bennequin surface $\mathcal{B}_T= \mathcal{B}_0\cup h_1\cup \dots \cup h_m$. 
\end{itemize}
In general, the twisted bands $h_j$ may intersect the interiors of the disks $\mathcal B_0$ transversely along embedded ribbon arcs. This happens exactly when in the abstract open book the arc $\alpha_j$ enters the collar neighbourhood $\Omega$, for some $j$.

In the case when there are no ribbon singularities, that is, none of the arcs $\alpha_j$ of the half Dehn twists passes through the collar neighbourhood $\Omega$, the braid $\beta$ is called \emph{strongly quasi-positive}. For strongly quasi-positive links some of the claims of this paper are easier to prove, the complication arise when: 1) the Bennequin surface has ribbon singularities, and 2) the braid monodromy does not satisfy assumption (*). To overcome the latter we use some auxiliary open books.

\subsection{Auxiliary open books} Given an abstract pointed open book $(S, P, \sigma, \phi)$ representing a braid $T$,  we can form an auxiliary braid $T_0$ in 
$\#^NS^1\times S^2$, where $N=b_1(S)=2g(S)+|\partial S|-1$. This is encoded by the pointed open book 
$(S,P, \sigma, \Id)$. 
The transverse link $T_0$ 
has an associated Bennequin surface since its monodromy can be arranged to satisfy (*) while in principle $T$ itself may not have a Bennequin surface. For example that cannot happen if $T$ is homologically essential.

\begin{remark}
If $S$ in $Y$ is an immersed surface arising as the image of some map $F\to Y$ we denote by $\chi(S)$ the embedded Euler characteristic of $S$ and by $\hat{\chi}(S)=\chi(F)$ the Euler characteristic of $S$ as an abstract surface.  
\end{remark}

The following lemma will be needed later in the paper. 

\begin{lemma}\label{lemma:curve} Let $C$ be a properly embedded pseudo-holomorphic curve in a Stein domain $(W,J)$. Suppose that $(S, P, \sigma , \phi)$ is a positive allowable open book describing $(M,\xi,T)=\partial (W,J,C)$. Then there is a factorisation of $\sigma$ that gives rise to a Bennequin surface 
$\mathcal B_{0}\subset \#^NS^1\times S^2$, with boundary the closure of the braid $T_0$, 
such that $\chi(C)=
\hat{\chi}(\mathcal B_{0})$.
\end{lemma}
\begin{proof}
Starting from the monodromy factorisation $\phi=D_{\gamma_1}\circ \cdots \circ D_{\gamma_s}$ in positive Dehn twists, we obtain a handle decomposition of the Stein manifold $W$ as a union $W=\natural^NS^1\times D^3 \cup Z$, where $Z$ is a collection of $2$-handles that are attached with contact framing along the $\gamma_i$ in the open book of $\#^NS^1 \times S^2$ associated  to $(S, \text{Id})$.

Following the proof of \cite[Theorem 4.14]{Hayden} we decompose $C$ as a union $C=\mathcal B\cup C'$, where $C'\subset Z$ consists of disjoint annuli, and $\mathcal B$ is a holomorphic curve in the standard Stein filling $\natural^NS^1\times D^3$ of $\#^NS^1\times S^2$. Indeed, since the boundary of $\mathcal{B}$ is presented by $(S,P,\sigma,\text{Id})$, we can identify $\partial \mathcal{B}$ with the auxiliary quasi-positive braid $T_0$ in $\#^NS^1\times S^2$, which bounds the Bennequin surface $\mathcal B_{0}$. 

Summarising, for said factorisation of $\sigma$ we have: $\hat{\chi}(\mathcal{B}_{0})=\chi(\mathcal{B})=\chi(C)$.
Indeed, $\hat{\chi}=|P|-m$ where $m$ is the number of positive half Dehn twists in the factorisation of $\sigma$ we chose in the construction of the two Bennequin surfaces.
\end{proof}

\section{The slice-Bennequin inequality}\label{Section:four}
To facilitate the arguments exposed in this section we shall use an  alternative definition of Hedden's $\tau_\xi$-invariant.

Recall that if $Y$ is a rational homology three-sphere, and $L$ in $Y$ is an $\ell$-component link, then we have a relation between the numbers of the $\F[U]$ towers in the collapsed link Floer homology group $cHFL^-(Y,L,\s)$  and the rank of $\widehat{HF}(Y,\s;\ell)$ for a given $\Spin^c$-structure $\s\in \text{Spin}^c(Y)$. Namely, there is a map 
\begin{equation}
    cHFL^-(Y,L,\s)\longrightarrow \widehat{HF}(Y,\s;\ell)\cong \widehat{HF}(Y,\s)\otimes_\F(\F_{(-1)}\oplus\F_{(0)})^{\otimes \ell-1}\:,
    \label{eq:Map}
\end{equation} 
which collapses each tower of $cHFL^-(Y,L,\s)$ into a distinct element of $\widehat{HF}(Y,\s)\otimes(\F_{(-1)}\oplus\F_{(0)})^{\otimes \ell-1}$, and sends an element with bigrading $(\text{gr},A)$ into one with Maslov grading $\text{gr}-2A$, see  \cite[Section 2]{Cavallo-L} for the details. 

Every tower in $cHFL^-(Y,L,\s)$ has a generator with bigrading of the form \[(\text{gr}(\alpha)+2A_{\alpha\otimes\Xi}+\text{gr}(\Xi),\:A_{\alpha\otimes\Xi})\] for each non-zero $\alpha\in\widehat{HF}(Y,\s)$ and $\Xi\in(\F_{(-1)}\oplus\F_{(0)})^{\otimes \ell-1}$.
Thus we have a total of $|\widehat{HF}(Y,\s)|$ towers whose generators have exactly bigrading $(\text{gr}(\alpha)+2A_\alpha,\:A_\alpha)$, corresponding to the choice $\Xi=\mathbf e_0$. Similarly, there are exactly $|\widehat{HF}(Y,\s)|$ towers with generators of bigrading $(\text{gr}(\alpha)+2A_{\alpha\otimes\e_{1-\ell}}+1-\ell,\:A_{\alpha\otimes\e_{1-\ell}})$, corresponding to the choice $\Xi=\mathbf e_{1-\ell}$. 
Note that $\tau_{\alpha}(L)=-A_{\alpha}$ and $\tau_{\alpha\otimes\mathbf e_{1-\ell}}(L)=-A_{\alpha\otimes\mathbf e_{1-\ell}}$, see \cite[Theorem 1.3]{Cavallo-tau}.

For a link $L$ in $Y$ and $\Sigma\in H_2(X,Y;\Z)$ where $X$ is a four-manifold with $\partial X=Y$, we recall that $\chi_4^X(L,\Sigma)$ is the maximal Euler characteristic of a compact, oriented surface $F$, properly embedded in $X$, and such that $\partial F=L$ and $[F]=\Sigma$.  
We want to prove the following inequality which generalises results from \cite{LM,Hedden-+,Plamenevskaya,Cavallo-fibered}.
\begin{teo}[Slice-Bennequin inequality]
 \label{teo:sB}
 Suppose that $(W,\omega)$ is a symplectic filling of a rational homology sphere $M$, equipped with a contact structure $\xi$. If $T$ is a transverse $\ell$-component link in $(M,\xi)$, then \[\self(T)\leq2\tau_{\xi}(T)-\ell\leq-\chi_4^{W}(T,\Sigma)-\Sigma\cdot\Sigma+c_1(\omega)(\Sigma)\:,\]
 where $\Sigma\in H_2(W,M;\Z)$.   
\end{teo}
We split the proof of Theorem \ref{teo:sB} into two parts, corresponding to the two inequalities appearing in the statement.

We recall that if $T$ is a transverse link in a contact three-manifold $(M,\xi)$ then we can consider the transverse invariant $\mathfrak L(T)\in cHFL^-(-M,T,\s_\xi)$ defined in \cite{LOSSz,Cavallo-L}. Such an invariant is non-torsion if and only if $\widehat c(\xi)\neq0$ (\cite[Theorem 1.2]{LOSSz}) and, whenever the Alexander grading is defined, we have the following relation \cite{OS,Cavallo-L}:

\begin{equation}    \label{eq:Grading}
    A(\mathfrak L(T))=\dfrac{\self(T)+|T|}{2}\ .
\end{equation} 

\begin{prop}
 \label{prop:one}
 Suppose that $(M,\xi)$ is a rational homology contact three-sphere with $\widehat c(\xi)$ non-vanishing, and $T$ is a transverse link in $(M,\xi)$. Then \[\self(T)\leq2\tau_\xi(T)-|T|\:.\]
\end{prop}
\begin{proof}
 Since by assumption $\widehat c(\xi)\in\widehat{HF}(-M,\s_\xi;\ell)$ is non-zero, we have that $\mathfrak L(T)$ is non-torsion and then \[\dfrac{\self(T)+|T|}{2}=A(\mathfrak L(T))\leq A_{\widehat c(\xi)}=\tau_\xi(T)\:.\] The first equality appears
 because of Equation \eqref{eq:Grading}, the inequality comes from the fact that the map in Equation \eqref{eq:Map} sends (the tower containing) $\mathfrak L(T)$ to $\widehat c(\xi)$ when the transverse invariant is non-vanishing \cite[Lemma 5.2]{Cavallo-L}, and the second equality follows from the definition of $\tau_\xi(T)$. 
\end{proof}

Before we establish the second part of the slice-Bennequin inequality, recall that a $\Spin^c$-cobordism  between two links $L_1$ and $L_2$ in $\Spin^c$ three-manifolds $(Y_1,\s_1)$ and $(Y_2, \s_2)$ is a triple $(X,F,\mathfrak u)$ where $X$ is a cobordism from $Y_1$ to $Y_2$,  $F$ in $X$ is a properly embedded, compact, oriented surface such that $\partial F=-L_1\sqcup L_2$, the connected components of $F$ have boundary on both $L_1$ and $L_2$, and $\mathfrak u$ is a $\Spin^c$-structure of $X$ restricting to $\s_1$ and $\s_2$ on the two ends.

\begin{prop}[Relative adjunction inequality]
 \label{prop:two}
 Suppose that $(X,F,\mathfrak u)$ is a $\Spin^c$-cobordism between $L_1\subset Y_1$ and $L_2\subset Y_2$, where $Y_i$ a rational homology three-sphere. 
 Assume that $\widehat F_{X,\mathfrak u}:\widehat{HF}_{d_1}(Y_1,\s_1;\ell_1)\rightarrow\widehat{HF}_{d_2}(Y_2,\s_2;\ell_2)$
 is such that $\widehat F_{X,\mathfrak u}(\alpha)=\beta$, where $\alpha$ and $\beta$ are non-zero.
 Then we have 
 \begin{equation}\tau_\beta(L_2)\leq\tau_\alpha(L_1)+g(F)+\ell_2-|F|-\dfrac{[F]\cdot[F]-c_1(\mathfrak u)[F]}{2}\:.\label{eqref:1}\end{equation}
 Furthermore, assume instead that $(\widehat F_{\overline X,\mathfrak u})^\bullet(\alpha')=\beta'$, where $\alpha'\in\widehat{HF}_{d_1+1-\ell_1}(Y_1,\s_1;\ell_1)$ and $\beta'\in\widehat{HF}_{d_2+1-\ell_2}(Y_2,\s_2;\ell_2)$ are non-zero. Then \begin{equation}\tau_{\beta'}(L_2)\leq\tau_{\alpha'}(L_1)+g(F)+\ell_1-|F|-\dfrac{[F]\cdot[F]-c_1(\mathfrak u)[F]}{2}\:,\label{eqref:2}\end{equation}
 where $\ell_i$ is the number of component of $L_i$.
\end{prop}
The proof of Proposition \ref{prop:two} appears in \cite[Theorems 1 and 5.17]{HR} by Hedden and Raoux. We just remark that the reason why the terms $g(F)+\ell_i-|F|$ appear in the right-side members is the filtered degree of the map $\widehat F$: decomposing the cobordism into canonical Morse pieces, we have that merge moves are filtered of degree 0 while split moves are filtered of degree 1 (this is reversed for the map in Equation \eqref{eqref:2}), and $g(F)+\ell_i-|F|$ is exactly the number of the latter ones. See \cite[Subsection 4.4]{Cavallo-upsilon} for more details.
We can now prove the slice-Bennequin inequality.

\begin{proof}[Proof of Theorem \ref{teo:sB}]
 Since $(M,\xi)$ is symplectically fillable, we have that $\widehat c(\xi)$ is non-zero \cite{Ghiggini2}. Therefore, we have that $\self(T)\leq2\tau_\xi(T)-|L|$ by Proposition \ref{prop:one}. 
    
 Now, the manifold $(W,J)$ is a symplectic cobordism 
 and induces the map $\widehat G_{\overline W,J}:=(\widehat F_{ W,J})^\bullet$, from the homology of the triple $(-M,T,\s_\xi)$ to the one of $(S^3,\bigcirc_{|F|})$ (the $|F|$-component unlink in $S^3$), which is non-zero \cite[Remark 2.14]{Ghiggini2}, as we observed in the proof of Theorem \ref{teo:supported}. 
 In other words, we have $\widehat G_{\overline W,J}(\widehat c(\xi))=\e_{1-|F|}$ and we can then use Proposition \ref{prop:two} to obtain \[\tau_\xi(T)=-\tau_{\widehat c(\xi)}(T)+\tau^*(\bigcirc_{|F|})\leq g(F)+|L|-|F|-\dfrac{[F]\cdot[F]-c_1(J)[F]}{2}\:,\]
 which means \begin{equation}
       2\tau_\xi(T)-|L|\leq-\chi(\hat F)-[\hat F]\cdot[\hat F]+c_1(J)[\hat F]
        \label{eq:sB}
 \end{equation} 
 because $\tau(\bigcirc_{|F|})=0$ and $\chi(\hat F)=2|F|-2g(F)-|L|$, where $\hat F$ is the capping of $F$ in $D^4$.
    
 The proof is concluded by observing that Equation \eqref{eq:sB} holds for every compact, oriented, properly embedded surface  $F'$ in $W$, with $\partial F'=T$. In particular, it holds for the surface that maximises the Euler characteristic in the relative homology class $\Sigma=[F']$ in $H_2(W,M;\Z)$.
\end{proof}
Note that, from \cite[Section 4]{Ghiggini}, we know that if $(W,J)$ is a Stein filling for $(M,\xi)$ then $(W,-J)$ is a Stein filling for $(M,\overline\xi)$, and the $\Spin^c$-structure induced by $\overline\xi$ is the conjugate of $\s_\xi$. If $C$ is a pseudo-holomorphic curve in $(W,J)$ then the same is true for $-C$ in $(W,-J)$. Therefore, we have the following symmetry property.   
\begin{cor}\label{cor:reverse}
 If $L$ is the boundary of a pseudo-holomorphic curve, properly embedded in $(W,J)$, then the link $-L$ obtained by reversing the orientation of $L$ is the boundary of a pseudo-holomorphic curve, properly embedded in $(W,-J)$, and $\tau_{\overline\xi}(-L)=\tau_\xi(L)$.  
\end{cor}
\begin{proof}
 It follows from Theorem \ref{teo:main} and the fact that $c_1(-J)=-c_1(J)$.
\end{proof}

We can now prove a topological version of the slice-Bennequin inequality.
\begin{proof}[Proof of Theorem \ref{teo:upper_bound}]
 Using the definition of the invariant $\widehat c(\xi)$ and $\theta=\widehat F_{W,J}(\textbf e_0)$ we have
 \[\dfrac{\self(L)+|T|}{2}\leq\tau_\xi(T)\leq\tau_\theta(T)\:.\]
 The first inequality comes from Proposition \ref{prop:one}, while the second one from Corollary \ref{cor:supported}.
\end{proof}

\section{Sharpness for Bennequin surfaces}\label{section:four}
We want to prove a partial version of Theorem \ref{teo:main} when $C$ is the positive Bennequin surface obtained from a quasi-positive braid encoded by the standard open book for $\#^NS^1\times S^2$. 

We start with a brief discussion regarding self-linking numbers. Recall that the self-linking number of a transverse link $T$ in $(S^3, \xi_{st})$ is defined as 
\begin{equation}\label{pippo}
\self(T)= w(\beta)-n 
\end{equation}
where $\beta\in B_n$ is a braid representing $T$ in the standard open book, and $w(\beta)$ denotes its writhe. This can be interpreted as the negative of the relative Euler number $e(\xi|_F,v)$ where  $F\subset M$ is a Seifert surface for $T$ in generic position with respect to $\xi$, and $v$ is a non-zero vector field  pointing in the direction of $\xi\cap TF$. There are two problems when extending this definition to a general contact three-manifold $(M,\xi)$: 
\begin{itemize}
\item Firstly, the Seifert surface $F$ may not exist. For this reason one must assume $T$ to be null-homologous in $M$ (or at least rationally null-homologous). 
\item Secondly, the right-hand side of Equation \ref{pippo} may depend on the choice of the Seifert surface. To avoid dependence in a rational homology sphere one can observe that if $F_1$ and $F_2$ are two surfaces with the same boundary link then $F=-F_1\cup F_2$ bounds a $3$-chain $\Delta$ and by Stokes theorem:
\[e(\xi|_{F_1}, v_1)-e(\xi|_{F_2}, v_2)= -\int_F c_1(\xi)= -\int_{\partial \Delta} c_1(\xi)=0 \:.\]  
\end{itemize}
If $\xi$ defines a torsion $\Spin^c$-structure on the other hand, $c_1(\xi)=0$ and still 
\[ e(\xi|_{F_1}, v_1)=e(\xi|_{F_2}, v_2) \ .\]
We conclude that there is no problem with the definition of the self-linking number $\self(T)$ of a transverse \emph{null-homologous} link braided around the standard open book of $(\#^NS^1\times S^2, \xi_0)$, where $\xi_0$ denotes the unique tight contact structure. Note that $\xi_0$ has a natural Stein filling $W=\natural^N S^1 \times D^3$.

\begin{prop}\label{prop:Bennequin} Let $T$ be the closure of a quasi-positive braid, encoded by the pointed open book $(S,P,\sigma,\Id)$ for $\#^N S^1\times S^2$, and $\mathcal B_T$ the Bennequin surface associated to some positive factorisation of $\sigma$ as in Section \ref{subsection:Bennequin}. Then 
 $\self(T)=-\hat{\chi}(\mathcal B_T)$. 
\end{prop}

\begin{proof}
The Bennequin surface $\mathcal{B}_T$ is the image of an immersion $f: F \looparrowright \#^N S^1\times S^2$ with $t$ arcs of ribbon singularities. From the construction of the Bennequin surface we know that $\chi(F)=\hat{\chi}(\mathcal{B}_T)=n-m$, where $m$ denotes the length of the factorisation of $\sigma$ we choose to make the construction, and $n=|P|$ the number of strands. 

We would like to use the Bennequin surface $\mathcal{B}_T$ to compute the self-linking number. If $T$ is strongly quasi-positive, $\mathcal{B}_T$ has no singularities and it has a Morse-Smale characteristic foliation. Because of the way the surface is positioned with respect to the open book, this has $e_+=n$  positive elliptic points, and $h_+=m$ positive hyperbolic points. Thus 
\[\self(T)=-(e_+-e_-)+(h_+-h_-) =m-n=-\hat{\chi}(\mathcal{B}_T) \ .\]

For an immersed surface on the other hand, we have no globally defined  characteristic foliation. Note that we can decompose $\mathcal{B}_T$ as a union of bands $H_1, \dots , H_m$ and disks $D_1, \dots, D_n$ and arrange those to have a well defined characteristic foliation. Indeed, we can make sure  that there is one positive hyperbolic point on each band $H_i$, and one positive elliptic point on each disk $D_j$, but we need a smooth surface to compute the self-linking number.

\begin{figure}[h]
 \centering
 \vspace{0.3cm}
  \def\svgwidth{0.75\textwidth}
\begingroup%
  \makeatletter%
  \providecommand\color[2][]{%
    \errmessage{(Inkscape) Color is used for the text in Inkscape, but the package 'color.sty' is not loaded}%
    \renewcommand\color[2][]{}%
  }%
  \providecommand\transparent[1]{%
    \errmessage{(Inkscape) Transparency is used (non-zero) for the text in Inkscape, but the package 'transparent.sty' is not loaded}%
    \renewcommand\transparent[1]{}%
  }%
  \providecommand\rotatebox[2]{#2}%
  \newcommand*\fsize{\dimexpr\f@size pt\relax}%
  \newcommand*\lineheight[1]{\fontsize{\fsize}{#1\fsize}\selectfont}%
  \ifx\svgwidth\undefined%
    \setlength{\unitlength}{757.92021984bp}%
    \ifx\svgscale\undefined%
      \relax%
    \else%
      \setlength{\unitlength}{\unitlength * \real{\svgscale}}%
    \fi%
  \else%
    \setlength{\unitlength}{\svgwidth}%
  \fi%
  \global\let\svgwidth\undefined%
  \global\let\svgscale\undefined%
  \makeatother%
  \begin{picture}(1,0.36295996)%
    \lineheight{1}%
    \setlength\tabcolsep{0pt}%
    \put(0,0){\includegraphics[width=\unitlength,page=1]{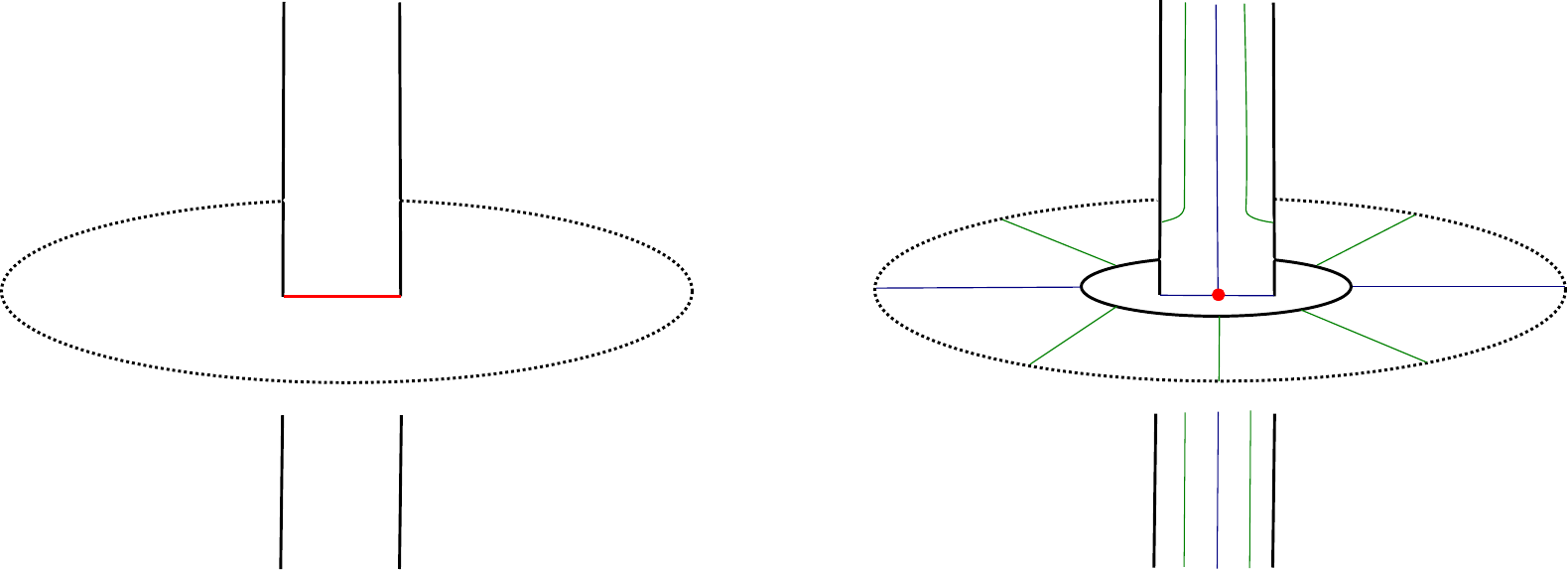}}%
    \put(0.78211349,0.18958089){\color[rgb]{0,0,0}\makebox(0,0)[lt]{\lineheight{1.25}\smash{\begin{tabular}[t]{l}$\mathfrak h_+$\end{tabular}}}}%
  \end{picture}%
\endgroup%

   \vspace{0.3cm}     
 \caption{\smaller[1]{The surface on the right-hand side is obtained by removing a negative elliptic point.}}
 \label{Foliation_A}
 \end{figure}

To overcome this problem we observe that each arc $\gamma \subset \mathcal{B}_T$ of ribbon singularities lifts to two distinct arcs in $F$. One of these arcs $\widetilde{\gamma}\subset F$ has endpoints lying in the interior of $F$, the other on $\partial F$.  For each of these arcs we choose a small disk $\Delta_\gamma \subset \text{int}(F)$ containing $\widetilde{\gamma}$ and we consider the image 
\[ \mathcal{B}_T'= f \left(F\setminus \bigcup_{\gamma } \text{int}(\Delta_\gamma )\right) \:, \]
which is now a surface with $t$ extra boundary components and no ribbon singularities. Note that for each arc of ribbon singularities $\gamma$ we can find a ball $B_\gamma \subset \#^N S^1\times S^2$ where the two surfaces $\mathcal{B}_T$ and $\mathcal{B}_T'$ differ as in Figure \ref{Foliation_A}. Of course the surface $\mathcal{B}_T'$ decomposes as the union of the bands $H_1, \dots, H_m$, and some punctured disks $D_1', \dots, D_n'$.  

Since $\mathcal{B}_T'$ is smooth we can consider its characteristic foliation $\mathcal{F}_\xi$, which can be arranged to have one positive hyperbolic point on each band $H_i$, one positive elliptic point on each punctured disk $D_i'$, and a total of $t$ negative hyperbolic points. To achieve this, we use Giroux's lemma to introduce cancelling pairs on the disks $D_i$ in order to get an isolated elliptic point at the centre of each disk $\Delta_\gamma$. We need to perform this manipulation once for each ribbon singularity for a total of $t$ times: each time we introduce a negative elliptic point, and a negative hyperbolic point. After the disks $\Delta_\gamma$ are removed to construct $\mathcal{B}_T'$ we see the $t$ negative elliptic points disappear and we are left with $t$ extra  negative hyperbolic points lying on $D_1'\cup \dots \cup D_n'$.

We now perform a further local modification: in each ball $B_\gamma$ we cut open the band $H_\gamma$ and attach the two loose ends to $\partial \Delta_i$  as shown in Figure \ref{Foliation_B}. Every time we do this we make the positive hyperbolic point on the band  disappear, and we create in exchange two new positive hyperbolic points for a net gain of one. That is we keep one positive elliptic point for each band and we introduce $t$ extra positive hyperbolic points, one for each ribbon singularity.

\begin{figure}[t]
 \centering
 \vspace{0.3cm}
  \def\svgwidth{0.35\textwidth}
\begingroup%
  \makeatletter%
  \providecommand\color[2][]{%
    \errmessage{(Inkscape) Color is used for the text in Inkscape, but the package 'color.sty' is not loaded}%
    \renewcommand\color[2][]{}%
  }%
  \providecommand\transparent[1]{%
    \errmessage{(Inkscape) Transparency is used (non-zero) for the text in Inkscape, but the package 'transparent.sty' is not loaded}%
    \renewcommand\transparent[1]{}%
  }%
  \providecommand\rotatebox[2]{#2}%
  \newcommand*\fsize{\dimexpr\f@size pt\relax}%
  \newcommand*\lineheight[1]{\fontsize{\fsize}{#1\fsize}\selectfont}%
  \ifx\svgwidth\undefined%
    \setlength{\unitlength}{335.53820512bp}%
    \ifx\svgscale\undefined%
      \relax%
    \else%
      \setlength{\unitlength}{\unitlength * \real{\svgscale}}%
    \fi%
  \else%
    \setlength{\unitlength}{\svgwidth}%
  \fi%
  \global\let\svgwidth\undefined%
  \global\let\svgscale\undefined%
  \makeatother%
  \begin{picture}(1,0.79761178)%
    \lineheight{1}%
    \setlength\tabcolsep{0pt}%
    \put(0,0){\includegraphics[width=\unitlength,page=1]{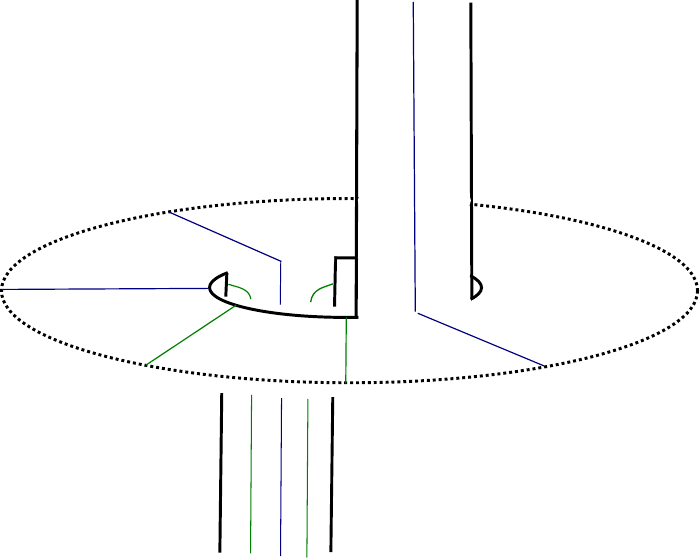}}%
    \put(0.39108132,0.45235507){\color[rgb]{0,0,0}\makebox(0,0)[lt]{\lineheight{1.25}\smash{\begin{tabular}[t]{l}$\mathfrak h_+$\end{tabular}}}}%
    \put(0.57004873,0.29662652){\color[rgb]{0,0,0}\makebox(0,0)[lt]{\lineheight{1.25}\smash{\begin{tabular}[t]{l}$\mathfrak h_+$\end{tabular}}}}%
    \put(0,0){\includegraphics[width=\unitlength,page=2]{Foliation-B.pdf}}%
  \end{picture}%
\endgroup%

  \vspace{0.3cm}      
 \caption{\smaller[1]{The characteristic foliation after removing the ribbon singularity.}}
 \label{Foliation_B}
 \end{figure}

Denote by $\mathcal{B}_T^*$ the surface obtained from $\mathcal{B}_T'$  via the local modification of Figure \ref{Foliation_B}, and by $\mathcal{F}_\xi^*$ its characteristic foliation. Since $\mathcal{B}_T^*$ is smooth and $\partial \mathcal{B}_T^*=T$  we can look at the singularities of $\mathcal{F}_\xi^*$ to compute $\self(T)$. It follows from our discussion that there are $e_+=n$ positive elliptic points (one on each punctured disk $D_i'$), $h_+=m+t$ positive hyperbolic points (distributed on the bands) and a total of $h_-= t$ negative hyperbolic points (distributed on the punctured disks). 
Thus 
\[ \self(T)=-(e_+-e_-)+(h_+-h_-)=-(n-0)+(m+t-t)=-\hat{\chi}(\mathcal B_T)\ .\]
\end{proof}

We can now prove Theorem \ref{teo:main} for a generic pseudo-holomorphic curve. 

\begin{proof}[Proof of Theorem \ref{teo:main}]
 To conclude we just need to show that for a properly embedded pseudo-holomorphic curve $C$, in a Stein filling $(W,J)$ with $\mathcal L=\partial C$, the slice-Bennequin inequality we established in Theorem \ref{teo:sB} is sharp. In other words, we need to show that \[\self(\mathcal L)=-\chi(C)-[C]\cdot[C]+c_1(J)[C]\:.\]
 By Lemma \ref{lemma:curve}, we know that $C=\mathcal B_T\sqcup_TC'$: here $T$ is the closure of a quasi-positive braid, with respect to the standard open book of the Stein fillable structure $\xi_0$ on $\#^NS^1\times S^2$, and $C'$ consists of $\ell$ disjoint embedded cylinders, where $\ell$ is the number of components of $\mathcal L$. 
 More precisely, we have a Stein concordance $(W',C',J')$ from $(\#^NS^1\times S^2,T,\s_{\xi_0})$ to $(M,\mathcal L,\s_\xi)$ induced by $(-1)$-contact surgery on a link. Moreover, \[c_1(J)[C]=c_1(J')[C']\hspace{1cm}\text{ and }\hspace{1cm}[C]\cdot[C]=[C']\cdot[C']\] because $H_2(\natural^NS^1\times D^3,\#^NS^1\times S^2;\Z)\cong\{0\}$. Proposition \ref{prop:Bennequin} then tells us that \[\self(T)=-\chi(\mathcal B_T)=-\chi(C)\:,\] implying that if we show that \begin{equation}
  \self(T)=\self(\mathcal L)+[C']\cdot[C']-c_1(J')[C']
  \label{eq:shift}
 \end{equation}
 then we are done. 
 We prove Equation \eqref{eq:shift} by directly computing how the self-linking number changes under contact surgery.
 We write $K=K^+\cup K^-\subset(S^3,\xi_{\text{st}})$ for the Legendrian links in the contact $(\pm\:1)$-surgery presentation of $(M,\xi)$, obtained by applying the construction in \cite{DG} to the pointed open book that represents $\mathcal L$ with respect to $(W,J)$. Hence, there is a transverse link $T_1\subset(S^3,\xist)$, disjoint from $K$, such that $T_0$ is gotten from it by applying contact $(+1)$-surgery along $K^+$.

 Writing $K=K_1\cup...\cup K_t$, let $a_i$ be the integral surgery coefficient on the link component $K_i$; i.e. $a_i=\tb(K_i)-1$ if $K_i\in K^-$ and $a_i=0$ if $K_i\in K^+$. Define the linking matrix \[Q_k(a_0,a_1,...,a_t)=(q_{i,j})_{i,j=0,...,t}\hspace{1cm}\text{ where }\hspace{1cm}
 q_{i,j}=\left\{\begin{aligned}
 &a_i\hspace{2.5cm}\text{ if }i=j \\
 &\lk(K_i,K_j)\hspace{1cm}\text{ if }i\neq j\end{aligned}\right.\:,\]
 with the convention that $T_1^k=K_0$, the $k$-th component of $T_1$, and $a_0=0$. Similarly, let $Q=Q(a_1, ...,a_t)$ denote the surgery matrix $(q_{i,j})_{i,j=1,...,t}$. We have \[[C']\cdot[C']=-\sum_{k=1}^\ell\dfrac{\det\left(Q_k(0,a_1,...,a_t)\right)}{\det Q}+2\sum_{a<b}\left\langle\left(\begin{matrix}
    \lk(T_1^a,K_1) \\ \vdots \\ \lk(T_1^a,K_t)
\end{matrix}\right),\:Q^{-1}\left(\begin{matrix}
    \lk(T_1^b,K_1) \\ \vdots \\ \lk(T_1^b,K_t)
\end{matrix}\right)\right\rangle\]
and \[c_1(J)[C']=\sum_{k=1}^\ell\left\langle\left(\begin{matrix}
    \rot(K_1) \\ \vdots \\ \rot(K_t)
\end{matrix}\right),\:Q^{-1}\left(\begin{matrix}
    \lk(T_1^k,K_1) \\ \vdots \\ \lk(T_1^k,K_t)
\end{matrix}\right)\right\rangle\:.\] 

 Let us call $\mathcal L_L$ and $T_L$ a Legendrian approximation for $\mathcal L$ in $(M,\xi)$ and $T_1$ in $(S^3,\xi_{\text{st}})$ respectively. We obtain
  \[\begin{aligned}
  \self(\mathcal L)
  &=\tb(\mathcal L_L)-\rot(\mathcal L_L)\\
  &=\sum_{k=1}^\ell\tb(\mathcal L^k_L)+2\sum_{a<b}\lk(\mathcal L^a_L,\mathcal L^b_L)-\sum_{k=1}^\ell\rot(\mathcal L^k_L)\\
  &=\sum_{k=1}^\ell\Bigg(\tb(T^k_L)+\dfrac{\det\left(Q_k(0,a_1,...,a_t)\right)}{\det Q}\Bigg)\\
  &+2\sum_{a<b}\left(\lk(T^a_L,T^b_L)-\left\langle\left(\begin{matrix}
    \lk(T_1^a,K_1) \\ \vdots \\ \lk(T_1^a,K_t)
 \end{matrix}\right),\:Q^{-1}\left(\begin{matrix}
    \lk(T_1^b,K_1) \\ \vdots \\ \lk(T_1^b,K_t)
 \end{matrix}\right)\right\rangle\right)\\
 &-\sum_{k=1}^\ell\left(\rot(T^k_L)-\left\langle\left(\begin{matrix}
    \rot(K_1) \\ \vdots \\ \rot(K_t)
\end{matrix}\right),\:Q^{-1}\left(\begin{matrix}
    \lk(T_1^k,K_1) \\ \vdots \\ \lk(T_1^k,K_t)
\end{matrix}\right)\right\rangle\right)\\
&=\self(T_1)-[C']\cdot[C']+c_1(J)[C'] \\
&=\self(T)-[C']\cdot[C']+c_1(J)[C']\:, 
 \end{aligned}\]  since the fact that $T$ is null-homologous implies $\self(T_1)=\self(T)$. The values of $\tb(\mathcal L_L)$ and $\rot(\mathcal L_L)$ are found using the formulae in \cite[Lemma 6.6]{LOSSz}.
\end{proof}

\begin{remark}
 In order to prove Equation \eqref{eq:shift}, we could also note that since $C'$ is a positive ascending surface \cite{Hayden} then the map induced in link Floer homology by the Stein concordance $(W',C',J')$ sends the transverse invariant $\mathfrak L(\mathcal L)$ onto $\mathfrak L(T)$. The shift of the Alexander grading, which determines the self-linking number from Equation \eqref{eq:Grading}, can then be computed using the results in \cite{JMZ}. Such an argument was communicated to us by Zemke.
 In the paper we decided to include the given proof since it only requires some linear algebra computations.   
\end{remark}


\section{Applications}\label{section:six}
\subsection{Proof of the main results} Finally, we prove Theorem \ref{teo:relative}, Proposition \ref{prop:relative} and Corollary \ref{cor:intro} from the introduction.
\begin{proof}[Proof of Theorem \ref{teo:relative}] Suppose we have some smoothly embedded surface $F\subset W$ bounding a link $T\subset M$ and that $T$ is    also the boundary of a pseudo-holomorphic curve $C$. By Theorem \ref{teo:main} and the relative adjunction inequality we have that \[-\chi(C) + c_1(J)[C]-    [C]^2=2\tau_\xi(T)-|T|\leq -\chi(F) + c_1(J)[F]-[F]^2 \ . \]  
 On the other hand, if $[F]=[C]$ in $H_2(W, M;\Z)$ then $ c_1(J)[C]-[C]^2= c_1(J)[F]-[F]^2$. Thus, we have that $\chi(C)\geq \chi(F)$ exactly as claimed.
\end{proof}
\begin{proof}[Proof of Proposition \ref{prop:relative} and Corollary \ref{cor:intro}]
 From Theorems \ref{teo:upper_bound} and \ref{teo:sB} we have the inequalities \[\self(T)\leq2\tau_\xi(T)-|T|\leq 2\tau_\theta(T)-|T|\leq -\chi(C) \] and we can run the same proof of Theorem \ref{teo:main} with the topological tau-invariant $\tau_\theta(T)$ instead. The second claim in the statements follows by combining the first one with \cite[Lemma 2.8]{MT} to get that $\tau_\xi(T)\leq\tau(M,T,\s)\leq\tau_\theta(T)$ is sharp.
\end{proof}

\subsection{The genus bound for links}
We observe that Lemma \ref{lemma:mirror} implies that for every $\Spin^c$ three-manifold $(Y,\s)$ and link $L\subset M$ one has
\begin{equation}
 \label{eq:mirror}
   \tau(-Y,L,\s)=-\tau^*(Y,L,\s)\:.
\end{equation}
As in the previous section, let $F$ denote a compact, oriented surface, properly embedded in a four-manifold $X$.

\begin{prop}  \label{prop:bound}
Suppose that $L$ is a link in a rational homology three-sphere $Y$, and that $L_0$ is a link in $S^3$. Moreover, assume that $(X,F)$ is a rational homology cobordism from $(L_0,S^3)$ to $(L,Y)$ with $|F|=|L_0|$. Then whenever the $\Spin^c$-structure $\s$ extends over $X$ we have that
 \begin{itemize}
     \item if $|L|=|L_0|$ then  $|\tau(Y,L,\s)-\tau(L_0)|\leq g(F)$;
     \item if $L_0$ is the unlink then $-g(F)\leq\tau(Y,L,\s)\leq g(F)+|L|-|F|$; in particular, one has $2\tau(Y,L,\s)-|L|\leq -\chi(\hat F)$, where $\hat F$ is the capping of $F$ in $D^4$.
 \end{itemize} 
\end{prop}
\begin{proof}
 Take a $\Spin^c$-structure $\mathfrak u$ on $X$ that extends $\s$. Since $X$ is a rational homology cobordism, one has that $X$ and $-X$ are negative-definite; hence, we can apply the inequalities in Proposition \ref{prop:two} with the classes $\theta$ and $\eta$, which are chosen as $\theta=\widehat F_{X,\mathfrak u}(\textbf e_0)$ and $\eta=\widehat G_{-X,\mathfrak u}(\textbf e_{1-|L_0|})=(\widehat F_{-\overline X,\uu})^\bullet(\textbf e_{1-|L_0|})$. Note that $\theta$ and $\eta$ are sent to $\Theta^+$ and $\Xi^+$ respectively, because of the commutation of the cobordism maps \cite{OSz-negative}. 
 Continuing the argument, Equation \eqref{eqref:2}, after reversing the orientation of $X$, and Equation \eqref{eq:mirror} give exactly \[-g(F)-\tau^*(L_0^*)\leq-\tau_{\eta}(L)\leq-\tau^*(-Y,L,\s)=\tau(Y,L,\s)\:,\] while Equation \eqref{eqref:1} gives \[\tau(Y,L,\s)\leq\tau_\theta(L)\leq\tau(L_0)+g(F)+|L|-|L_0|\:.\] Since $L_0\subset S^3$ one has $\tau^*(L_0^*)=-\tau(L_0)$ from which the first inequality in the statement follows. For the second one, we argue in the same way as before and, since $\tau(\bigcirc_{|L_0|})=0$, we obtain \[-g(F)\leq\tau(Y,L,\s)\leq g(F)+|L|-|L_0|\] which implies the claim.
\end{proof}

Proposition \ref{prop:bound} extends \cite[Proposition 4.7]{Cavallo-tau} and \cite[Theorem 3.6]{GRS}, see also \cite[Theorem 1.5]{Cavallo-upsilon}.

\subsection{Obstructions}
\begin{proof}[Proof of Corollary \ref{obs2}]  
 Suppose that $L$ bounds a pseudo-holomorphic curve $C$ in $(W,J)$.
 From Theorem \ref{teo:main} and Corollary \ref{cor:intro} we have \[\tau(M,L,\s_\xi)=\tau_\xi(L)=\dfrac{|L|-\chi_4^W(L)}{2}=g(C)+|L|-|C|\:.\] From Proposition \ref{prop:bound} we obtain \[-\tau(M,L,\s_\xi)=-g(C)-|L|+|C|\leq-g(C)\leq\tau(M,L,\mathfrak t)\leq g(C)+|L|-|C|=\tau(M,L,\s_\xi)\] and then $|\tau(M,L,\mathfrak t)|\leq\tau(M,L,\s_\xi)$ for every $\Spin^c$-structure $\mathfrak t$ that extends over $W$.
\end{proof}
In particular, when a quasi-positive $\ell$-component link $L$ in $M$ satisfies $\tau_\xi(L)=0$, and bounds a pseudo-holomorphic curve in a rational homology ball Stein filling $(W,J)$ of $(M,\xi)$, we have that $L$ is forced to be rationally slice in $W$. This is because of Theorem \ref{teo:main}: one has \[-\ell=2\tau_\xi(L)-\ell=-\chi^W_4(L)\hspace{1cm}\text{ which implies  }\hspace{1cm}\chi_4^W(L)=\ell\:,\] and this is the maximum possible value for $\chi^W_4(L)$ that can only be obtained by a surface consisting of $\ell$ disjoint disks. 
From Proposition \ref{prop:bound} we have that $L$ has a metaboliser $G$ such that $\tau(M,L,\mathfrak \s_\xi+\alpha)=0=\tau(M,L,\s_\xi)$ for every $\alpha\in G$. This holds also without the assumption on $c^+(\xi)$.
\begin{proof}[Proof of Corollary \ref{obs1}]
 If $L$ bounds a pseudo-holomorphic curve when perturbing both $J$ and $-J$ then we should have \[\tau(M,L,\s_\xi)=\tau_\xi(L)=\tau_{\overline\xi}(L)=\tau(M,L,\overline{\s_\xi})\:,\] but this is a contradiction.
\end{proof}

\subsection{A bound for the piece-wise slice genus of a link} 
The strong piece-wise linear slice genus $\mathfrak g_{PL}$ is defined as the minimal genus of a PL-cobordism between two links, such that each of its connected component is a PL-concordance of knots. A looser definition can be given by considering every PL-cobordism instead, and such an invariant is denoted by $g_{PL}$. In this paper we work with $\mathfrak g_{PL}$ in light of the examples in Section \ref{section:seven}. 

\begin{proof}[Proof Theorem \ref{teo:PL}]Suppose that the surface $F$ has minimal genus within all $PL$ surfaces with $\ell=|L|$ connected components  each bounding a different component of the link $L$. 

Let $F_1, \dots , F_\ell$ denote the connected components of the surface $F$. For each component $F_i$  choose a piece-wise linear path $\gamma_i\subset F_i$ passing through the PL singularities on the component, and some small tubular neighbourhood    $\gamma_i\subset N(\gamma_i)$ so that:
\begin{itemize}
\item $N(\gamma_i)$ contains all the singularities of $F_i$;
\item $N(\gamma_i) \cap N(\gamma_j)= \emptyset$ for $i \not= j$;
\item $N(\gamma_i)\cap F_i$ in some knot $K_i\subset \partial N(\gamma_i) \simeq S^3$. 
\end{itemize}
 
Piping $N(\gamma_1), \dots , N(\gamma_\ell)$ together we get a smooth four-ball $B\subset X$ whose complement $Z=X\setminus B$ contains a smooth cobordism $F^*=F\cap Z$ from some link $L_0$ in $S^3$ with $\ell$ components  to the link  $L$ in $Y$. Of course $g(F)=g(F^*)$, and $F^*$ has $\ell$ connected components  $F_1^*, \dots , F_\ell^*$ each providing a cobordims from some component of $L_0$ to some component of $L$.
 
 Now let $\s_1$ and $\s_2$ be two $\text{Spin}^c$-structures on $Y$ extending over $X$ such that $\tau(Y,L, \s_1)=\tau_{\max}^X(Y,L)$, and $\tau(Y,L, \s_2)=\tau_{\min}^X(Y,L)$. Using Proposition \ref{prop:bound} we obtain two inequalities
 \[|\tau(Y,L,\s_1)-\tau(L_0)|\leq g(F^*)\hspace{1cm}\text{ and }\hspace{1cm}|\tau(L_0)-\tau(Y,L,\s_2)|\leq g(F^*)\:.\] 
 Summing the two inequalities, and applying the triangle inequality we get the desired inequality.
 
 Geometrically this is the same as saying that the two $\text{Spin}^c$-structures  $\s_1$ and $\s_2$ have extension $\mathfrak{u}_1$ and $\mathfrak{u}_2$ that restrict to the unique $\Spin^c$-structure of $S^3$ on the left-hand of the cobordism $Z$. Thus there is $\text{Spin}^c$-structure $\mathfrak{u}$ on the double  $W=-Z\cup_Y Z $ interpolating between $\s_1$ and $\s_2$. Using  Proposition \ref{prop:bound} we get an inequality
\[|\tau(Y,L,\s_1)-\tau(Y,L,\s_2)|\leq g(-F^*\cup_{L_0}F^*)=2g(F) \ ,\]
exactly as claimed.
\end{proof}

Since in general $\mathfrak g_{\text{PL}}(L)\leq g_4(L)$, we have that the lower bound in Theorem \ref{teo:PL} holds also for the smooth slice genus. Although, the reader should note that such a bound is usually worse compared to the one given by Proposition \ref{prop:two}; in fact, for \emph{local links}, i.e. links embedded in a 3-ball in $Y$, all the $\tau$'s have the same value and then Theorem \ref{teo:PL} does not give any information. 
Moreover, we have the following corollary.
\begin{cor} 
 \label{cor:concordance}
 If a link $L$ as before is such that $\tau^X_{\max}(Y,L)\neq\tau^X_{\min}(Y,L)$ then it is not rational homology concordant to a local link in any rational homology three-sphere.
\end{cor}

\subsection{Invariants of covering} We now prove Theorem \ref{teo:double}.

\begin{proof}[Proof of Theorem \ref{teo:double}] If $K$ in $S^3$ is a quasi-positive knot and $C\subset D^4$ is a holomorphic curve bounding $K$ then as consequence of the work of Loi and Piergallini \cite{LP} the branched double-cover of $D^4$ along $C$ is a Stein domain $(W,J)$, bounding the double-cover $\Sigma(K)$ of $S^3$ branching along $K$. Indeed, the fixed point set of the covering involution $\sigma: W\to W$ is a complex curve $\widetilde{C}\subset W$ with $\partial\widetilde{C} = \widetilde{K}$ the pull-back knot. We denote by $\xi_0$ the contact structure induced by $J$ on $\partial W=\Sigma(K)$.

Let $F$ be a Seifert surface for $K$ in $S^3$. Since $H_3(D^4;\Z)=0$ we can find a $3$-chain $M\subset D^4$ with $\partial M= C-F$. This lifts to two distinct $3$-chains $\widetilde{M}$ and $\sigma(\widetilde{M})$ overlapping over $\widetilde{C}$, furthermore: $\widetilde{M}=\widetilde{C} -\widetilde{F}$
where $\widetilde{F}$ is one of the two lifts of the Seifert surface $F$ to the branched double-covering $\Sigma(K)$. Consequently, we have that $\widetilde{C}=\partial \widetilde{M}$ in the relative homology group $H_3(W, \partial W ; \Z)$, showing that:
\[\tau_{\xi_0}(\widetilde{K})=-\dfrac{\chi(\widetilde{C})-1-c_1(J)[ \widetilde{C}] + [\widetilde{C}]^2}{2}=\dfrac{1 -\chi(\widetilde{C})}{2}=\dfrac{1-\chi(C)}{2}=\tau(K) \ . \]  
where of course the first identity is due to Theorem \ref{teo:main}.

Since $K$ is a quasi-alternating knot we can conclude \cite{MO} that $\Sigma(K)$ is an Heegaard Floer $L$-space, thus $\tau_{\xi_0}(\widetilde{K})= \tau(\widetilde{K},\s_0)$ where $\s_0$ denotes the $\text{Spin}^c$-structure associated to the homotopy class of $\xi_0$. 

Finally, we need to identify $\s_0$ with the spin structure of $\Sigma(K)$. To this end we observe that: the $\text{Spin}^c$-structure $J$ is self-conjugate since we can run the same argument with the opposite orientation, but the branched double-covering construction in \cite{LP} does not depend on orientations. On the other hand, for a knot $K$ in $S^3$ we have that $H_1(\Sigma(K); \Z/2\Z)=0$, so there is no 2-torsion in $H^2(W;\Z)$ and $J$ is forced to be the unique spin structure supported by $W$. Clearly, it follows that $\s_0$ is spin, being the restriction of $J$ to $\Sigma(K)$.
\end{proof}

\section{Examples}\label{section:seven}
In this section we denote $\tau(Y,L,\s)$ by $\tau(L,\s)$, omitting the name of the three-manifold. 
\subsection{A family of quasi-positive links in \texorpdfstring{$\mathbf{\emph{L}(4,1)}$}{}}
The lens space $L(4,1)$ bounds a rational homology ball $W$ given by the complement of a generic quadric $Q$ in $\C P^2$. Using the Veronese embedding $W \to \C^6$, one can obtain a Stein fillable contact structure $\xi$ on $L(4,1)$. A degree $d$ smooth curve $\overline{C}$ intersects a generic quadric $Q \subset \C P^2$ transversely in $2d$ points. Thus $C=\overline{C} \setminus Q= \overline{C} \cap W$ is a holomorphic curve intersecting $\partial W = L(4,1)$ in a link $L_{2d}$ with $2d$ components.

We have that $L_{2d}$, see Figure \ref{L_2d}, bounds a connected holomorphic curve $C$ in $(W,J)$ with genus $\frac{(d-1)(d-2)}{2}$. 
This implies that \[2\tau_\xi(L_{2d})-2d=-\chi(C)=(d-1)(d-2)+2d-2=d^2-d\:,\] and then \[\tau_\xi(L_{2d})=\dfrac{d(d+1)}{2}\] where $\xi$ is the Stein fillable structure on $L(4,1)$ induced by $(W,J)$.
According to the symmetry under mirroring and the fact that $L(4,1)$ is an $L$-space, we have shown that $\tau(L_{2d},\s_\xi)=\frac{d(d+1)}{2}$. Note that, since $W$ is rational homology ball, the $\Spin^c$-structure $\s_\xi$ should be one of the two (conjugate) structures with correction term zero. 
\begin{prop}
 \label{prop:L_2d}
 We have that \[\left|\tau_{\max}(L_{2d})-\tau_{\min}(L_{2d})\right|=d\:,\] and then $\mathfrak g_{\text{PL}}(L_{2d})\geq\lceil\frac{d}{2}\rceil>0$. In particular, $L_{2d}$ is not rational homology concordant to any link in $S^3$.    
\end{prop}
\begin{proof}
 Since $L(4,1)$ is a lens space, we have that $W$ admits $\Spin^c$-structures which restrict to $\sqrt{|H_1(L(4,1);\Z)|}=2$ $\Spin^c$-structures on $L(4,1)$: in this case those structures are already identified as $\s_\xi$ and its conjugate $\overline{\s_\xi}$. Hence, there are two possible $\tau$'s: one is $\tau_\xi(L_{2d})$, which we already computed, and the other one is $\tau(L_{2d},\overline{\s_\xi})$.

 In order to compute $\tau(L_{2d},\overline{\s_\xi})$ we observe that the link $L_{2d}$ bounds holomorphic curves also in the Stein filling $D_{-4}$ displayed in Figure \ref{L(4,1)}. More precisely, such a four-manifold (which is not a rational homology ball) has three Stein structures, two of which, call them $J_1$ and $J_2$ as in Figure \ref{L(4,1)}, restricts to $\s_\xi$ and $\overline{\s_\xi}$ respectively.
 \begin{figure}[t]
 \centering
 \vspace{0.6cm}
  \def\svgwidth{6cm}
\begingroup%
  \makeatletter%
  \providecommand\color[2][]{%
    \errmessage{(Inkscape) Color is used for the text in Inkscape, but the package 'color.sty' is not loaded}%
    \renewcommand\color[2][]{}%
  }%
  \providecommand\transparent[1]{%
    \errmessage{(Inkscape) Transparency is used (non-zero) for the text in Inkscape, but the package 'transparent.sty' is not loaded}%
    \renewcommand\transparent[1]{}%
  }%
  \providecommand\rotatebox[2]{#2}%
  \newcommand*\fsize{\dimexpr\f@size pt\relax}%
  \newcommand*\lineheight[1]{\fontsize{\fsize}{#1\fsize}\selectfont}%
  \ifx\svgwidth\undefined%
    \setlength{\unitlength}{500.02974863bp}%
    \ifx\svgscale\undefined%
      \relax%
    \else%
      \setlength{\unitlength}{\unitlength * \real{\svgscale}}%
    \fi%
  \else%
    \setlength{\unitlength}{\svgwidth}%
  \fi%
  \global\let\svgwidth\undefined%
  \global\let\svgscale\undefined%
  \makeatother%
  \begin{picture}(1,0.32791402)%
    \lineheight{1}%
    \setlength\tabcolsep{0pt}%
    \put(0,0){\includegraphics[width=\unitlength,page=1]{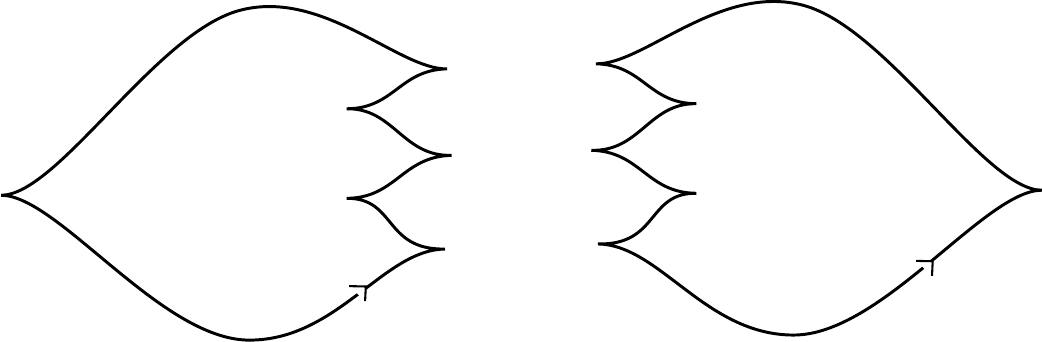}}%
    \put(-0.00123013,0.25852426){\color[rgb]{0,0,0}\makebox(0,0)[lt]{\lineheight{1.25}\smash{\begin{tabular}[t]{l}$(-1)$\end{tabular}}}}%
    \put(0.91339638,0.25421761){\color[rgb]{0,0,0}\makebox(0,0)[lt]{\lineheight{1.25}\smash{\begin{tabular}[t]{l}$(-1)$\end{tabular}}}}%
  \end{picture}%
\endgroup%

 \vspace{0.3cm}
 \caption{\smaller[1]{The Stein structures $J_1$ and $J_2$ on the Stein filling $D_{-4}$ of $L(4,1)$. Here $(-1)$ denotes contact surgery (smooth framing $\tb-1$).}}
 \label{L(4,1)}
 \end{figure}
 This means that we can use Theorem \ref{teo:main} to compute the value of $\tau$-invariant.
 
 The link $L_{2d}$ bounds a holomorphic curve $C'$ consisting of $2d$ meridional disks in $D_{-4}$. We then have \[2\tau(L_{2d},\s_i)-2d=-2d-[C']\cdot[C']+c_1(J_i)[C']\] and thus \[\tau(L_{2d},\s_i)=-\dfrac{[C']\cdot[C']-c_1(J_i)[C']}{2}\:.\]
 We can compute $c_1(J_i)[C']$ and $[C']\cdot[C']$ using the formulae in Section \ref{section:four}: \[[C']\cdot[C']=2d\cdot-\dfrac{\det\left(\begin{matrix}
  0 & 1\\
  1 & -4
 \end{matrix} \right)}{\det (-4)}+2\cdot\binom{2d}{2}\cdot(1)\left(-\dfrac{1}{4}\right)(1)=-\dfrac{d}{2}-\dfrac{d(2d-1)}{2}=-d^2\] and \[c_1(J_i)[C']=2d\cdot(\pm 2)\left(-\dfrac{1}{4}\right)(1)=\pm d\:.\] Depending on whether $i=1$ or $i=2$ we have \[\tau(L_{2d},\s_i)=\dfrac{d(d+1)}{2}\hspace{1cm}\text{ and }\hspace{1cm}\tau(L_{2d},\s_i)=\dfrac{d(d-1)}{2}\:,\] but since the first value coincides with the one of $\tau_\xi(L_{2d})$ it follows that \[\tau(L_{2d},\overline{\s_\xi})=\dfrac{d(d-1)}{2}\:.\]
 The second part of the statement now follows from Theorem \ref{teo:PL} and Corollary \ref{cor:concordance}.
\end{proof}

\subsection{A family of quasi-positive links in \texorpdfstring{$\mathbf{\emph{L}(9,2)}$}{L(9,2)}}
\label{subsection:7.2}
Start with a degree $d$ smooth curve $Z\subset \C P^2$. Pick a generic line $L$ and a point $p\in \C P^2 \setminus Z\cup L$. Blowing-up at $p$ we get a smooth four-manifold $F_1=\C P^2 \# \overline{\C P^2}$ together with a well defined projection $\pi:F_1\to L\simeq\C P^1$ whose fibres are the strict transforms of the lines through $p$. This is the first Hirzebruch surface, see Figure \ref{Nagata} (left). We denote by $E_1=L'$ the strict transform of $L$ and by $E_{-1}$ the exceptional divisor of the blow-up. Obviously $Z_1=Z'$, the strict transform of $Z$, does not intersect $E_{-1}$ and does intersect $E_1$ in $d$ points. 
Furthermore, $Z_1$ intersects the generic fibres of the projection $F_1\to \C P^1$ in $d$ points.

We now want to pass from the first Hirzebruch surface $F_1\to \C P^1$ to the second Hirzebruch surface $F_2\to \C P^1$. To do so we first blow-up at the intersection $q$ of $E_{-1}$ with a generic fibre $f=\pi^{-1}(t)$, and then we blow-down the strict transform $f'$ of the fibre. The resulting birational transformation 

\[
\begin{tikzcd}
 &\arrow{dl} (F_1)'\arrow{dr} \\
F_1  \arrow{rr}{\nu} && F_1
\end{tikzcd}
\]
is the so called Nagata transform, see Figure \ref{Nagata} (right). We would like to understand the push-forward $Z_2= \nu_*(Z_1)=\overline{\nu(Z_1\setminus p)}$ of the curve $Z_1$.  

\begin{figure}[t]
 \centering
 \vspace{0.6cm}
  \def\svgwidth{0.82\textwidth}
        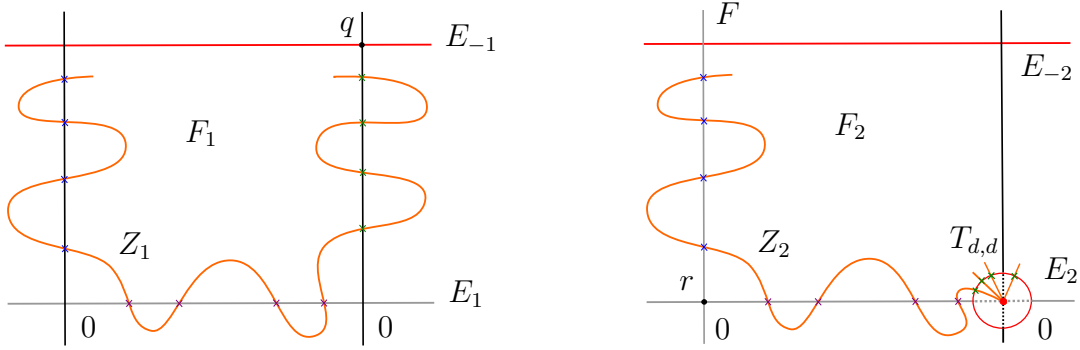
  \vspace{0.3cm}      
 \caption{\smaller[1]{The Nagata transform of the Hirzebruch surface $F_1$.}}
 \label{Nagata}
 \end{figure}

 \begin{figure}[t]
\vspace{0.9cm}
 \centering
  \def\svgwidth{0.6\textwidth}
        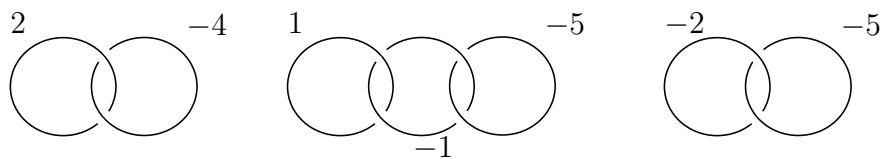
 \vspace{0.3cm}       
        \caption{\smaller[1]{A diffeomorphism of $L(9,2)$.}}
 \label{Move}
\end{figure}

First we observe that $Z_1'$, the strict transform of $Z_1'$ trough the  blow-up $(F_1)'\to F_1$, is a smooth genus $g=(d-1)(d-2)/2$ curve intersecting $f'$ in $d$ points. So when we blow-down $(F_1)'\to F_2$  to perform the Nagata transform these $d$ points get identified together transforming $Z''$ into a singular curve $\tau_*(Z)$. Note that the strict transforms $E_1'$ and $E_{-1}'$ correspond through $(F_1)'\to F_2$ to sections $E_2$ and $E_{-2}$ with self-intersection $\pm2$. 
Furthermore, the push-forward $F$ of the exceptional divisor $E$ of the blow-up $(F_1)'\to F_1$ is a fibre of the projection $F_2\to \C P^1$ intersecting $Z_2$ with multiplicity $d$ in its unique singular point. Summarising, we have the following identities:  
\[F\cdot E_2 = F \cdot E_{-2}=1,   \ \ \ F^2=0 , \ \ \  E_2\cdot E_2=+2,  \ \ \  E_{-2}\cdot E_{-2}=-2, \ \ \  Z_2\cdot F=d \ .\]
Using the singular curve $Z_2\subset F_2$ we produce a  link $M_{3d}$  with $3d$ components in $L(9,2)$. To do so we choose a generic fibre of $\pi: F_2 \to \C P^1$ and we observe that after desingularising the unique double point $r$ of $E_{2}\cup \pi^{-1}(t)$ we get a smooth sphere $S_{+4}\subset F_2$ with self-intersection $+4$. This intersects $Z_2$ in $2d+1$ points (one singular, $2d$ smooth),  $E_{-2}$ in one point, and a neighbourhood of $E_{-2}\cup S_4$ gives a smooth embedding $P_{\Gamma}\hookrightarrow F_2$ of the plumbing $P_\Gamma$ with graph
 \[\begin{tikzpicture}[scale=0.3]
    \tkzDefPoints{0/0/-2, 5/0/4}
    \tkzDrawPolygon(-2,4)\tkzDrawPoints(-2,4)
    \tkzLabelPoints[above left](-2)
    \tkzLabelPoints[above right](4)
  \end{tikzpicture} \:.\]
The complement of said embedding is a rational homology ball $W=F_2 \setminus \nu(S_{+4}\cup E_{-2})$ in a Stein domain with boundary $\partial W=L(9,2)$. We define $M_{3d}$ as the intersection of the curve $Z_2$ with the hyper-surface $L(9,2) \simeq \partial W\subset F_2$. To make a drawing of the link $M_{3d}$ we look at the configuration $E_{-1}'\cup S_{+4}'\cup f'$ in $(F_1)'$. Here $S_{+4}'$ denotes the strict transform of $S_{+4}$ through the blow-up $(F_1)'\to F_2$. We note that the curve $Z_1'\subset (F_1)'$ does not intersect $E_{-1}'$, it intersects $S_{+4}'$ in $2d$ points, and it intersects the fibre $F$ in $d$ points.

So $M_{3d}$ is presented by the diagram
in Figure \ref{M_3d}, where the diffeomorphism $ \partial W\to  \partial W$ in the picture is the one induced by the blow-up $(F_1)' \to F_2$ in Figure \ref{Move}.

We have that $M_{3d}$ bounds a connected holomorphic curve $C=Z_2\cap W$ in $(W,J)$ with genus $\frac{(d-1)(d-2)}{2}$. This implies that \[2\tau_\xi(M_{3d})-3d=-\chi(C)=(d-1)(d-2)+3d-2=d^2\:,\] and then \[\tau_\xi(M_{3d})=\frac{d(d+3)}{2}\] where $\xi$ is the Stein fillable structure on $L(9,2)$ induced by $(W,J)$.
Since $L(9,2)$ is a lens space, we have shown that $\tau(M_{3d},\s_\xi)=\frac{d(d+3)}{2}$. 

As in the previous subsection, the spin structure $\s_\xi$ should be one of the two (conjugate) structures with correction term zero, but in this case since $|H_1(L(9,2);\Z)|=9$ we also have a self-conjugate $\Spin^c$-structure $\mathfrak s_0$, induced by the unique spin structure on $L(9,2)$. While $\s_\xi$ and $\overline{\s_\xi}$ are both restriction of a Stein structure, the structure $\mathfrak s_0$ does not even come from a tight structure on $L(9,2)$.

There are three possible $\tau$'s: one of them is $\tau_\xi(M_{3d})$, which we already computed, and the other two are $\tau(M_{3d},\mathfrak s_0)$ and $\tau(M_{3d},\overline{\s_\xi})$.
We cannot compute $\tau(M_{3d},\overline{\s_\xi})$ as we did for $\tau(L_{2d},\overline{\s_\xi})$ because the link $M_{3d}$ does not necessarily bound a holomorphic curve in the Stein filling $D_{-5,-2}$ displayed in Figure \ref{L(9,2)}.

Nevertheless, such a four-manifold has four Stein structures, two of which, call them $J_1$ and $J_2$ as in Figure \ref{L(9,2)}, restricts to $\s_\xi$ and $\overline{\s_\xi}$ respectively. This means that we can use Theorem \ref{teo:main} to estimate the value of $\tau(M_{3d},\overline{\s_\xi})$.

A simple computation \cite{LOSSz} shows that $\s_{J_i}$ for $i=1,2$, corresponding to the two Stein structures in which we are interested, is determined by the Chern class $c_1(J_i)=(\pm3,0)$. 
The link $M_{3d}$ has a transverse representative with self-linking number equal to $-d^2-2d$  and bounds a (smooth) curve $C'$ consisting of $2d$ meridional disks and a connected component of genus $\frac{(d-1)(d-2)}{2}$ in $D_{-5,-2}$. We then have  
\[2\tau(M_{3d},\s_i)-3d=-2d+(d-1)(d-2)+d-2-[C']\cdot[C']+c_1(J_i)[C']\] and thus \[\dfrac{-d^2+d}{2}\leq\tau(M_{3d},\s_i)+\dfrac{[C']\cdot[C']-c_1(J_i)[C']}{2}\leq\dfrac{d^2-d}{2}\:.\]
We again compute $c_1(J_i)[C']$ and $[C']\cdot[C']$ using the formulae in Section \ref{section:four}: 

\begin{figure}[t]
 \centering
  \def\svgwidth{8.5cm}
\begingroup%
  \makeatletter%
  \providecommand\color[2][]{%
    \errmessage{(Inkscape) Color is used for the text in Inkscape, but the package 'color.sty' is not loaded}%
    \renewcommand\color[2][]{}%
  }%
  \providecommand\transparent[1]{%
    \errmessage{(Inkscape) Transparency is used (non-zero) for the text in Inkscape, but the package 'transparent.sty' is not loaded}%
    \renewcommand\transparent[1]{}%
  }%
  \providecommand\rotatebox[2]{#2}%
  \newcommand*\fsize{\dimexpr\f@size pt\relax}%
  \newcommand*\lineheight[1]{\fontsize{\fsize}{#1\fsize}\selectfont}%
  \ifx\svgwidth\undefined%
    \setlength{\unitlength}{608.61961064bp}%
    \ifx\svgscale\undefined%
      \relax%
    \else%
      \setlength{\unitlength}{\unitlength * \real{\svgscale}}%
    \fi%
  \else%
    \setlength{\unitlength}{\svgwidth}%
  \fi%
  \global\let\svgwidth\undefined%
  \global\let\svgscale\undefined%
  \makeatother%
  \begin{picture}(1,0.29415754)%
    \lineheight{1}%
    \setlength\tabcolsep{0pt}%
    \put(0,0){\includegraphics[width=\unitlength,page=1]{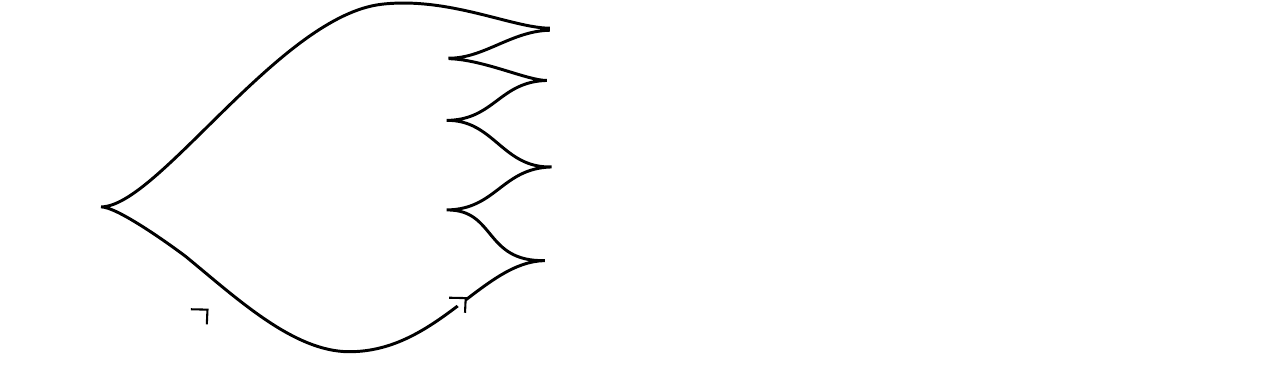}}%
    \put(0.14762268,0.28128076){\color[rgb]{0,0,0}\makebox(0,0)[lt]{\lineheight{1.25}\smash{\begin{tabular}[t]{l}$(-1)$\end{tabular}}}}%
    \put(0.02422246,0.00353822){\color[rgb]{0,0,0}\makebox(0,0)[lt]{\lineheight{1.25}\smash{\begin{tabular}[t]{l}$(-1)$\end{tabular}}}}%
    \put(0.77886233,0.27014213){\color[rgb]{0,0,0}\makebox(0,0)[lt]{\lineheight{1.25}\smash{\begin{tabular}[t]{l}$(-1)$\end{tabular}}}}%
    \put(0.92527672,0.01996441){\color[rgb]{0,0,0}\makebox(0,0)[lt]{\lineheight{1.25}\smash{\begin{tabular}[t]{l}$(-1)$\end{tabular}}}}%
    \put(0,0){\includegraphics[width=\unitlength,page=2]{L_9,2_.pdf}}%
  \end{picture}%
\endgroup%
       
 \caption{\smaller[1]{The Stein structures $J_1$ and $J_2$ on the Stein filling $D_{-5,-2}$ of $L(9,2)$.}}
 \label{L(9,2)}
\end{figure} 

\[\begin{aligned}\relax
[C']\cdot[C']&=3d\cdot-\dfrac{\det\left(\begin{matrix}
  0 & 0 & 1 \\
 0 & -2 & 1 \\
  1 & 1 & -5
 \end{matrix} \right)}{\det\left(\begin{matrix}
  -2 & 1 \\
   1 & -5
 \end{matrix} \right)}+2\cdot\binom{3d}{2}\cdot\left(\begin{matrix}1 & 0\end{matrix}\right)\left(\begin{matrix}
     -\dfrac{2}{9} & -\dfrac{1}{9} \\ \\
    -\dfrac{1}{9} & -\dfrac{5}{9}
 \end{matrix}\right)\left(\begin{aligned} &1 \\ &0\end{aligned}\right)= \\ 
 &=-\dfrac{2d}{3}-\dfrac{6d(3d-1)}{9}=-2d^2
 \end{aligned}\] and \[c_1(J_i)[C']=3d\cdot\left(\begin{matrix}\pm 3 & 0\end{matrix}\right)\left(\begin{matrix}
     -\dfrac{2}{9} & -\dfrac{1}{9} \\ \\
     -\dfrac{1}{9} & -\dfrac{5}{9}
 \end{matrix}\right)\left(\begin{matrix}1 \\ 0\end{matrix}\right)=\pm 2d\:.\]

Comparing the value of $\tau_\xi(M_3)$ with what we obtained yields 
$c_1(J_1)[C']=-2d$ and $c_1(J_2)[C']=2d$. Hence, we can write
\[\dfrac{d(d-1)}{2}\leq\tau(M_{3d},\overline{\s_\xi})\leq\dfrac{3d(d-1)}{2}\:.\]
We are ready to prove the following proposition.
\begin{prop}
 We have that $\left|\tau_{\max}(M_{3d})-\tau_{\min}(M_{3d})\right|>0$, and then $\mathfrak g_{\text{PL}}(M_{3d})>0$ for $d=1,2$. In particular, $M_{3}$ and $M_6$ are not rational homology concordant to any link in $S^3$.    
\end{prop}
\begin{proof}
 Since $L(9,2)$ is a lens space, all three $\Spin^c$-structures on $L(9,2)$, with vanishing correction term, extends to $W$. 
 In order to get positive difference between $\tau_\xi(M_{3d})$ and $\tau(M_{3d},\overline{\s_\xi})$ we need $\frac{3d(d-1)}{2}<\frac{d(d+3)}{2}$, which translates to $d^2<3d$. This implies that when $d=1,2$ we have \[\left|\tau_{\max}(M_{3d})-\tau_{\min}(M_{3d})\right|\geq\left|\dfrac{3d(d-1)}{2}-\dfrac{d(d+3)}{2}\right|=3d-d^2>0\] and the statement follows from Theorem \ref{teo:PL} and Corollary \ref{cor:concordance}.
\end{proof}

In the case that $d=1$ we can also compute $\tau(M_{3},\mathfrak s_0)$, but we need to use the methods of lattice cohomology. In \cite{Alfieri} the first author came up with a formula to compute the upsilon invariant of a graph knot, see \cite[Section 2]{Alfieri} for details. A similar formula can  be shown to be true for any link in a rational graph. This is a fairly simple adaptation of the argument in \cite[Section 4.4]{Alfieri}. 

\begin{teo}\label{teo:Antonio}
 If $L$ is an $\ell$-component link described by some unmarked leaves of a negative-definite almost-rational tree $G$, then 
 \[\tau(Y(G),L,\s)=\frac{1}{2}\left(\min_{k} k \cdot \Sigma -\Sigma\cdot\Sigma\right) \] 
 where the minimum is taken over all characteristic vectors $k\in H_2(X(G);\Z)$ for the intersection form of the associated plumbing $X(G)$ restricting to the $\Spin^c$-structure $\s$, and $\Sigma$ is the obvious collection of  disk fibres in the filling $X(G)$ bounding $L$.
\end{teo}

We apply this theorem to our setting. We take $v=(a\: b)$ as a covector. Then \[0=d(\mathfrak s)=\dfrac{v^2[X]-3\sigma(X)-2b_2(X)}{4}\] which implies \[-2=v^2[X]=\left(\begin{matrix} a & b \end{matrix}\right)\left(\begin{matrix}
     -\dfrac{2}{9} & -\dfrac{1}{9} \\ \\
     -\dfrac{1}{9} & -\dfrac{5}{9}
 \end{matrix}\right)\left(\begin{matrix} a \\ b \end{matrix}\right)=-\dfrac{1}{9}\cdot\left(\begin{matrix} a & b \end{matrix}\right)\left(\begin{matrix} 2a+b \\ a+5b \end{matrix}\right)=-\dfrac{1}{9}\cdot(2a^2+2ab+5b^2)\] and therefore we obtain \begin{equation}2a^2+2ab+5b^2=18\:.
  \label{eq:quadric}
\end{equation}
We now compute $\tau(M_{3},\mathfrak s_0)$. The only integer vectors which satisfy Equation \eqref{eq:quadric} are $(\pm\:3,0)$ and $(\pm\:1,\mp2)$, but the first pair corresponds to Stein fillable structures. Hence, we are left with the second pair which yields to $v\cdot\Sigma=0$ regardless of $\Sigma$.
We then conclude that \[\tau(M_3,\mathfrak s_0)=-\dfrac{[C']\cdot[C']}{2}=1\] and we can write 

\begin{equation}
  \tau(M_3,\s_i)=\left\{\begin{aligned}
    &2\hspace{1cm}\text{ for }i=1 \\
    &1\hspace{1cm}\text{ for }i=0 \\
    &0\hspace{1cm}\text{ for }i=2
  \end{aligned}  \right. \:.
  \label{eq:M_3}
\end{equation}

Note that the computation in Equation \eqref{eq:M_3} agrees with the one performed by Celoria with his computer program. 
For the final result of the paper we consider the link $N_k$ in $L(9,2)$ shown in Figure \ref{N_k}. Such a link bounds disjoint holomorphic disks in the Stein structures $J_1$ and $J_2$ in Figure \ref{L(9,2)}. Hence, the same procedure used for $M_{3d}$ and Theorem \ref{teo:Antonio} yields: \[\tau(N_k,\s_i)=\left\{\begin{aligned}
    &\dfrac{k^2+3k}{9}\:\hspace{1cm}\text{ for }i=1 \\
    &\dfrac{k^2}{9}\hspace{2cm}\text{ for }i=0 \\
    &\dfrac{k^2-3k}{9}\:\hspace{1cm}\text{ for }i=2
  \end{aligned}  \right. \:.\]


\begin{figure}[t]
 \centering
  \def\svgwidth{4.5cm}
\begingroup%
  \makeatletter%
  \providecommand\color[2][]{%
    \errmessage{(Inkscape) Color is used for the text in Inkscape, but the package 'color.sty' is not loaded}%
    \renewcommand\color[2][]{}%
  }%
  \providecommand\transparent[1]{%
    \errmessage{(Inkscape) Transparency is used (non-zero) for the text in Inkscape, but the package 'transparent.sty' is not loaded}%
    \renewcommand\transparent[1]{}%
  }%
  \providecommand\rotatebox[2]{#2}%
  \newcommand*\fsize{\dimexpr\f@size pt\relax}%
  \newcommand*\lineheight[1]{\fontsize{\fsize}{#1\fsize}\selectfont}%
  \ifx\svgwidth\undefined%
    \setlength{\unitlength}{289.94725745bp}%
    \ifx\svgscale\undefined%
      \relax%
    \else%
      \setlength{\unitlength}{\unitlength * \real{\svgscale}}%
    \fi%
  \else%
    \setlength{\unitlength}{\svgwidth}%
  \fi%
  \global\let\svgwidth\undefined%
  \global\let\svgscale\undefined%
  \makeatother%
  \begin{picture}(1,0.56724245)%
    \lineheight{1}%
    \setlength\tabcolsep{0pt}%
    \put(0,0){\includegraphics[width=\unitlength,page=1]{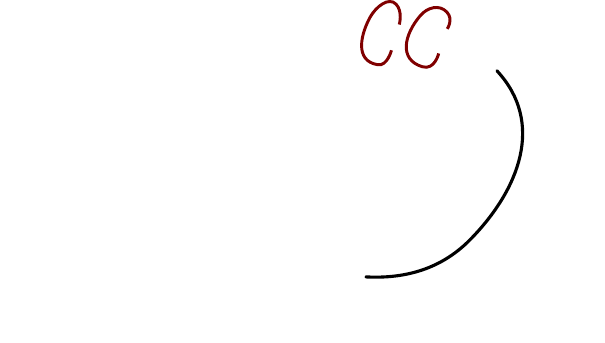}}%
    \put(0.88991444,0.13626852){\color[rgb]{0,0,0}\makebox(0,0)[lt]{\lineheight{1.25}\smash{\begin{tabular}[t]{l}$N_k$\end{tabular}}}}%
    \put(0,0){\includegraphics[width=\unitlength,page=2]{N_k.pdf}}%
    \put(0.18187297,0.00656823){\color[rgb]{0,0,0}\makebox(0,0)[lt]{\lineheight{1.25}\smash{\begin{tabular}[t]{l}$-2$\end{tabular}}}}%
    \put(0.61569922,0.01165216){\color[rgb]{0,0,0}\makebox(0,0)[lt]{\lineheight{1.25}\smash{\begin{tabular}[t]{l}$-5$\end{tabular}}}}%
    \put(0,0){\includegraphics[width=\unitlength,page=3]{N_k.pdf}}%
  \end{picture}%
\endgroup%
       
 \caption{\smaller[1]{The $k$-component link $N_k$ in $L(9,2)$.}}
 \label{N_k}
 \end{figure}

\begin{proof}[Proof of Proposition \ref{prop:last}]
 We apply Corollary \ref{obs2} and the fact that $\tau(N_{3d},\s_2)$ is not maximal to obstruct the existence of a pseudo-holomorphic curve in $(W,-J)$. For the case of $(W,J)$ we just observe that when $k$ is not a multiple of 3, the value of $\tau(N_{k},\s_1)$ is not an integer, and this would contradict Theorem \ref{teo:main}. 
\end{proof}

\end{document}